\documentclass[a4paper,11pt,english]{smfart}

\usepackage{amssymb}
\usepackage{url}
\usepackage[T1]{fontenc}
\usepackage[cal=boondoxo]{mathalfa}

\usepackage{fancybox}
\usepackage{amsmath}
\usepackage{makecell}

\usepackage[leftbars]{changebar}
\usepackage{todonotes}

\usepackage{palatino}
\usepackage{rotating}
\usepackage{graphicx}

\usepackage{enumerate}
\usepackage{enumitem}

\usepackage{color}
\definecolor{shadecolor}{gray}{0.90}

\input xypic
\xyoption{all}
\xyoption{arc}

\makeindex

\newtheorem{theo}{Theorem}[section]

\newtheorem{prop}[theo]{Proposition}

\newtheorem{lem}[theo]{Lemma}

\newtheorem{coro}[theo]{Corollary}

\def\equat{\refstepcounter{theo}\begin{equation}}
\def\endequat{\end{equation}}

\renewcommand\thesection{\arabic{section}}

    \def\AM{{\mathbb{A}}}
    
    \def\CM{{\mathbb{C}}}

  \def\mG{{\mathfrak m}}  
    \def\NM{{\mathbb{N}}}
    
  \def\pG{{\mathfrak p}}  \def\PM{{\mathbb{P}}}
    \def\QM{{\mathbb{Q}}}
    \def\RM{{\mathbb{R}}}

    \def\ZM{{\mathbb{Z}}}

    \def\AC{{\mathcal{A}}}
    \def\BC{{\mathcal{B}}}
    \def\CC{{\mathcal{C}}}
    
\def\Eb{{\mathbf E}}    \def\EC{{\mathcal{E}}}
  \def\fb{{\mathbf f}}

    \def\IC{{\mathcal{I}}}

\def\Lb{{\mathbf L}}

    \def\OC{{\mathcal{O}}}

\def\Sb{{\mathbf S}}    \def\SC{{\mathcal{S}}}
    
    \def\UC{{\mathcal{U}}}
    \def\VC{{\mathcal{V}}}
    \def\WC{{\mathcal{W}}}
    \def\XC{{\mathcal{X}}}
    \def\YC{{\mathcal{Y}}}
    \def\ZC{{\mathcal{Z}}}

  \def\drm{{\mathrm{d}}}  
\def\Erm{{\mathrm{E}}}    
    
\def\Grm{{\mathrm{G}}}    
\def\Hrm{{\mathrm{H}}}    
\def\Irm{{\mathrm{I}}}

\def\Lrm{{\mathrm{L}}}

\def\Orm{{\mathrm{O}}}

\def\Srm{{\mathrm{S}}}    
\def\Trm{{\mathrm{T}}}

    \def\CCt{{\tilde{\mathcal{C}}}}

    \def\VCt{{\tilde{\mathcal{V}}}}
    
    \def\XCt{{\tilde{\mathcal{X}}}}

    \def\CCh{{\hat{\mathcal{C}}}}
    \def\DCh{{\hat{\mathcal{D}}}}

  \def\phat{{\hat{p}}}

    \def\XCh{{\hat{\mathcal{X}}}}

\def\a{\alpha}
\def\b{\beta}

\def\d{\delta}
\def\D{\Delta}
\def\e{\varepsilon}
\def\ph{\varphi}
\def\l{\lambda}
\def\L{\Lambda}

\def\o{\omega}

\def\r{\rho}
\def\s{\sigma}

\def\th{\theta}

\def\z{\zeta}

        \def\pht{{\tilde{\varphi}}}

\def\chib{{\boldsymbol{\chi}}}

\def\mub{{\boldsymbol{\mu}}}

           \def\Delh{{\hat{\Delta}}}
      
           \def\phh{{\hat{\varphi}}}

               \def\pih{{\hat{\pi}}}
               
             \def\rhoh{{\hat{\rho}}}

\DeclareMathOperator{\diag}{{\mathrm{diag}}}

\DeclareMathOperator{\End}{{\mathrm{End}}}

\DeclareMathOperator{\Id}{{\mathrm{Id}}}

\DeclareMathOperator{\Irr}{{\mathrm{Irr}}}
\DeclareMathOperator{\Ker}{{\mathrm{Ker}}}

\DeclareMathOperator{\val}{{\mathrm{val}}}

\def\to{\rightarrow}
\def\longto{\longrightarrow}
\def\injto{\hookrightarrow}

\def\fonction#1#2#3#4#5{\begin{array}{rccc}
{#1} : & {#2} & \longto & {#3}  \\
& {#4} & \longmapsto & {#5}
\end{array}}
\def\fonctionl#1#2#3#4#5{\begin{array}{rccl}
{#1} : & {#2} & \longto & {#3} \\
& {#4} & \longmapsto & {#5}
\end{array}}
\def\fonctio#1#2#3#4{\begin{array}{ccc}
{#1} & \longto & {#2} \\
{#3} & \longmapsto & {#4}
\end{array}}

\def\vide{\varnothing}
\def\DS{\displaystyle}

\def\SSS{\scriptscriptstyle}

\def\finl{~$\blacksquare$}

\def\lexp#1#2{\kern\scriptspace\vphantom{#2}^{#1}\kern-\scriptspace#2}
\def\le{\hspace{0.1em}\mathop{\leqslant}\nolimits\hspace{0.1em}}
\def\ge{\hspace{0.1em}\mathop{\geqslant}\nolimits\hspace{0.1em}}

\mathchardef\inferieur="321E
\mathchardef\superieur="321F

\def\eqna{\begin{eqnarray*}}
\def\endeqna{\end{eqnarray*}}

\def\itemth#1{\item[${\mathrm{(#1)}}$]}

\catcode`\@=11
\long\def\@car#1#2\@nil{#1}
\long\def\@first#1#2{#1}
\long\def\@second#1#2{#2}
\long\def\ifempty#1{\expandafter\ifx\@car#1@\@nil @\@empty
  \expandafter\@first\else\expandafter\@second\fi}
\catcode`\@=12

\def\GL{{\mathrm{GL}}}

\newcommand{\REF}{{\mathrm{Ref}}}

\def\boitegrise#1#2{\begin{centerline}{\fcolorbox{black}{shadecolor}{~
    \begin{minipage}[t]{#2}{\vphantom{~}#1\vphantom{$A_{\DS{A_A}}$}}
            \end{minipage}~}}\end{centerline}\medskip}

\theoremstyle{remark}
\newtheorem{rema}[theo]{Remark}

\newtheorem{exemple}[theo]{Example}

\theoremstyle{plain}

\def\BIL{LR}
\def\GAUCHE{L}
\def\CAR{CAR}
\def\FAM{FAM}

\def\xyinj{\ar@{^{(}->}}
\def\xysur{\ar@{->>}}

\bigskip\def\unb{{\boldsymbol{1}}}

\def\petitespace{\vphantom{$\DS{\frac{\DS{A^A}}{\DS{A_A}}}$}}

\makeatletter
\def\hlinewd#1{%
\noalign{\ifnum0=`}\fi\hrule \@height #1 %
\futurelet\reserved@a\@xhline}
\makeatother

\newlength\epaisLigne

\usepackage{multirow}

\def\dotcup{\hskip1mm\dot{\cup}\hskip1mm}

\newcommand{\longiso}{\stackrel{\sim}{\longrightarrow}}

\makeatletter
\def\hlinewd#1{%
\noalign{\ifnum0=`}\fi\hrule \@height #1 %
\futurelet\reserved@a\@xhline}
\makeatother

\usepackage{multirow}

\def\petitespace{\vphantom{$\DS{\frac{\DS{A^A}}{\DS{A_A}}}$}}

\usepackage{lscape}

\def\setw{{\mathrm{set}}}
\def\ptw{{\mathrm{pt}}}

\usepackage{array, boldline, makecell, booktabs}
\usepackage{multirow,multicol}

\newtheorem{comp}[theo]{Computation}

\def\multiset#1{\{\!\{#1\}\!\}}

\def\petitespace{\vphantom{\DS{\frac{\DS{A}}{\DS{A}}}}}

\def\setw{{\mathrm{set}}}
\def\ptw{{\mathrm{pt}}}
\def\Pic{{\mathrm{Pic}}}
\def\Sym{{\mathrm{Sym}}}

\def\GL{\Grm\Lrm}
\def\Deg{{\mathrm{Deg}}}
\def\Codeg{{\mathrm{Codeg}}}

\def\trois{{\SSS{3}}}
\def\quatre{{\SSS{4}}}

\begin{document}
\title[Weyl group of type $\Erm_6$ and K3 surfaces]{Weyl group of type $\Eb_{\boldsymbol{6}}$ and K3 surfaces}

\author{C\'edric Bonnaf\'e}
\address{IMAG, Universit\'e de Montpellier, CNRS, Montpellier, France}

\email{cedric.bonnafe@umontpellier.fr}
\urladdr{http://imag.umontpellier.fr/~bonnafe/}
\makeatother

\date{\today}

\begin{abstract}
Adapting methods of previous papers by A. Sarti and the author, we construct
K3 surfaces from invariants of the Weyl group of type $\Erm_6$. We study in details
one of these surfaces, which turns out to have Picard number $20$: for this
example, we describe an elliptic fibration (and its singular fibers),
the Picard lattice and the transcendental lattice.
\end{abstract}

\maketitle


This paper can be seen as a continuation of previous works by A. Sarti and the author on
the construction of {\it singular} K3 surfaces by using invariants of finite reflection
groups~\cite{bonnafe sarti 1,bonnafe sarti 2}. We consider here
quotients of complete intersections defined by fundamental invariants of
a group of rank $6$, namely the Weyl group of type $\Erm_6$.

Let $W$ be the Weyl group of type $\Erm_6$ acting on $V=\CM^6$ and let $W'$
denote its derived subgroup (which has index $2$ in $W$ and
is equal to $W \cap \Sb\Lb_\CM(V)$). Then
the algebra $\CM[V]^W$ of polynomial functions on $V$ invariant under the action
of $W$ is generated by six homogeneous and algebraically independent
polynomials $f_2$, $f_5$, $f_6$, $f_8$, $f_9$ and $f_{12}$ (of respective
degrees $2$, $5$, $6$, $8$, $9$ and $12$).

We denote by $\XC \subset \PM(V)$ the surface defined by $f_2=f_6=f_8=0$ and,
for $\l$, $\mu \in \CM^2$, we denote by $\YC_{\l,\mu}$ the surface defined
by $f_5=f_6+\l f_2^3 = f_8+\mu f_2^4=0$. It turns out that $\XC$ is smooth
and that $\YC_{\l,\mu}$ is smooth for generic values of $(\l,\mu)$.
Our first main result in this paper is the following (see Theorem~\ref{theo:k3}):

\bigskip

\noindent{\bf Theorem A.}
{\it
\begin{itemize}
\itemth{a} The minimal resolution of the singular surface $\XC/W'$
is a smooth K3 surface.

\itemth{b} If $\l$, $\mu$ be such that $\YC_{\l,\mu}$ has at most
ADE singularities\footnote{We do not know if there are values  of $(\l,\mu)$ such that
$\YC_{\l,\mu}$ has more complicated singularities.}, then the minimal
resolution of $\YC_{\l,\mu}/W'$ is a smooth K3 surface.
\end{itemize}}

\bigskip

In the rest of the paper, we investigate further properties of $\XC/W'$ and of its
minimal resolution $\XCt$. The second main result of the paper is the following
(see Theorem~\ref{theo:elliptic}):

\bigskip

\noindent{\bf Theorem B.} {\it
\begin{itemize}
\itemth{a} The K3 surface $\XCt$ admits an elliptic fibration
$\pht : \XCt \to \PM^1(\CM)$ whose singular fibers are of type $E_7 + E_6 + A_2 + 2\, A_1$.

\itemth{b} The Picard lattice of $\XCt$ has rank $20$ and discriminant $-228=-2^2 \cdot 3 \cdot 19$.

\itemth{c} The transcendental lattice of $\XCt$ is given by the matrix
$\begin{pmatrix} 2 & 0 \\ 0 & 114 \end{pmatrix}$.
\end{itemize}}

\bigskip

In the course of the proof of Theorem~B, we obtain a complete description of
the Picard lattice: it is generated by $22$ smooth rational curves whose intersection
graph is given by\footnote{This means that two smooth rational curves $C$ and $C'$ in
this list intersect if and only if
they are joined by an edge in the above graph, and that, if so, then $C \cdot C'=1$.}
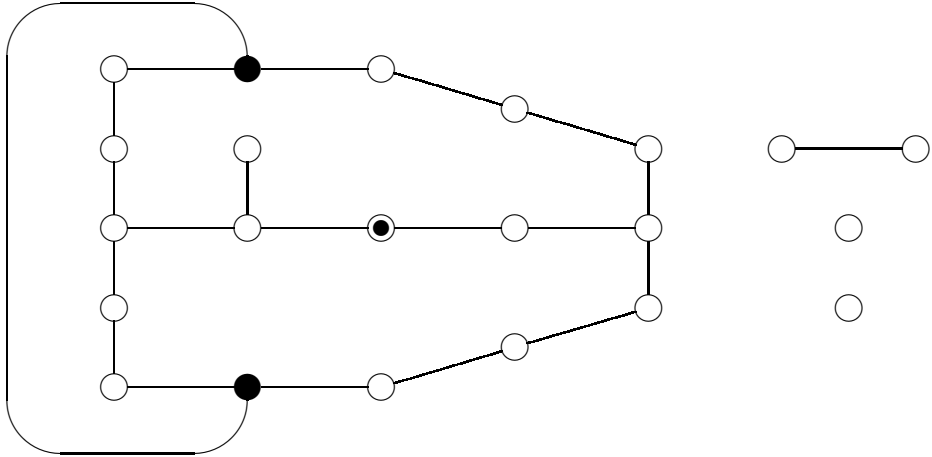
\centerline{
\begin{picture}(360,175)
\put( 55,140){\circle{10}}
\put(105,140){\circle*{10}}
\put(155,140){\circle{10}}
\put(205,125){\circle{10}}
\put(255,110){\circle{10}}
\put( 55,110){\circle{10}}
\put(105,110){\circle{10}}
\put( 55, 80){\circle{10}}
\put(105, 80){\circle{10}}
\put(155, 80){\circle{10}}\put(155,80){\circle*{6}}
\put(205, 80){\circle{10}}
\put(255, 80){\circle{10}}
\put( 55, 50){\circle{10}}
\put( 55, 20){\circle{10}}
\put(105, 20){\circle*{10}}
\put(155, 20){\circle{10}}
\put(205, 35){\circle{10}}
\put(255, 50){\circle{10}}

\put(305,110){\circle{10}}
\put(355,110){\circle{10}}
\put(330,80){\circle{10}}
\put(330,50){\circle{10}}
\put( 60, 140){\line(1,0){40}}
\put( 55, 135){\line(0,-1){20}}
\put(110, 140){\line(1,0){40}}
\put( 55, 105){\line(0,-1){20}}
\put( 55, 75){\line(0,-1){20}}
\put( 55, 45){\line(0,-1){20}}
\put(105, 105){\line(0,-1){20}}
\put(255, 105){\line(0,-1){20}}
\put(255, 75){\line(0,-1){20}}
\put( 60, 80){\line(1,0){40}}
\put(110, 80){\line(1,0){40}}
\put(160, 80){\line(1,0){40}}
\put(210, 80){\line(1,0){40}}
\put( 60, 20){\line(1,0){40}}
\put(110, 20){\line(1,0){40}}

\put(310,110){\line(1,0){40}}

\qbezier(159.789,138.563)(180,132.5)(200.211,126.437)
\qbezier(209.789,123.563)(230,117.5)(250.211,111.437)
\qbezier(159.789,21.437)(180,27.5)(200.211,33.563)
\qbezier(209.789,36.563)(230,42.5)(250.211,48.437)

\put(60,140){\oval(90,50)[t]}
\put(60,20){\oval(90,50)[b]}
\put(15,140){\line(0,-1){120}}


\end{picture}}

\bigskip

\noindent
In this graph, the two black disks correspond to sections of the elliptic
fibration $\pht$ mentioned in Theorem~B(a). The disk with a smaller black disk inside
corresponds to a double section. The singular fibers
of $\pht$ of type $E_7$ and $E_6$ are given by the white disks
of the big connected subgraph.

\medskip

The paper is organized as follows. In Section~\ref{sec:preliminaries},
we fix general notation and prove some preliminary results about group actions
on tangent spaces. In Section~\ref{sec:e6}, we fix the context, recall properties
of the Weyl group of type $E_6$ and recall results from Springer theory~\cite{springer}.
Section~\ref{sec:k3} is mainly devoted to the proof of Theorem~A. Section~\ref{sec:xw}
gathers many geometric properties of the quotient variety $\XC/W'$ (singularities,
smooth rational curves, explicit equations in a weighted projective space):
in a first reading of this paper, we believe the reader might skip the details
of this very computational section and remember
only the concrete results. Theorem~B is proved in Section~\ref{sec:xtilde}.
As a complement to all these data,
the last section~\ref{sec:action} contains the decomposition of the character of the
representation $\Hrm^2(\XC,\CM)$ of $W$ into a sum of irreducible characters
(note that $\dim \Hrm^2(\XC)=9\,502$).

Some of the proofs require computer calculations, especially for all the
preliminary results for proving Theorem~B: they are done
thanks to the software {\sc magma}~\cite{magma} and are
explicitly described in the Appendices.

\bigskip

\section{General notation, preliminaries}\label{sec:preliminaries}

\medskip

All vector spaces, all algebras, all algebraic varieties
will be defined over $\CM$. Algebraic varieties will always
be reduced and quasi-projective, but not necessarily irreducible.
If $\XC$ is an algeraic variety and if $x \in \XC$, we denote
by $\Trm_x(\XC)$ the tangent space of $\XC$ at $x$. If $\XC$ is
moreover affine, we denote by $\CM[\XC]$ its algebra of regular functions.

We fix a square root $i$ of $-1$ in $\CM$.
If $V$ is a vector space, $g$ is an element of $\End_\CM(V)$
and $\z \in \CM$, we denote by $V(g,\z)$ the $\z$-eigenspace of $g$.
The list of eigenvalues of an element of $\End_\CM(V)$ will always
be given with multiplicities (and will be seen as a multiset: a multiset
will be always written with double brackets, e.g. $\multiset{a,b,\dots}$).
If $d \in \NM^*$, we denote
by $\mub_d$ the group of $d$-th roots of unity in $\CM^\times$ and we set $\z_d=\exp(2i\pi/d)$.

If $V$ is a finite dimensional vector space and if $v \in V \setminus \{0\}$, we denote by $[v]$ its image in
the projective space $\PM(V)$. If $p \in \PM(V)$, we denote by $G_p$ its stabilizer in $G \subset \GL_\CM(V)$.
In other words, $G_{[v]}$ is the set of elements of $G$ admitting $v$ as an eigenvector.

If $X$ is a subset of $V$ and if $G$ is a subgroup of $\GL_\CM(V)$, we denote by
$G_X^\setw$ (resp. $G_X^\ptw$) the setwise (resp. pointwise) stabilizer of $X$ and
we set $G[X]=G_X^\setw/G_X^\ptw$ (so that $G[X]$ acts faithfully on the set $X$ and
on the vector space generated by $X$).

If $d_1$,\dots, $d_n$ are positive integers, we denote by $\PM(d_1,\dots,d_n)$
the associated weighted projective space. If $f_1$,\dots, $f_r \in \CM[X_1,\dots,X_n]$ are homogeneous
(with $X_i$ endowed with the degree $d_i$), we denote by $\ZC(f_1,\dots,f_r)$ the (possibly
non-reduced) closed subscheme defined by $f_1=\cdots=f_r=0$.

The next three lemmas will be used throughout this paper. The first one is
trivial, but very useful~\cite[Lem.~2.2]{bonnafe sarti 1}:

\bigskip

\begin{lem}\label{lem:trivial}
Let $V$ be a finite dimensional vector space,
let $f \in \CM[V]$, let $g \in \GL_\CM(V)$ and let $v \in V \setminus \{0\}$ be such that:
\begin{itemize}
\itemth{1} $g(v)=\xi v$, with $\xi \in \CM^\times$.

\itemth{2} $f$ is homogeneous of degree $d$ with $\xi^d \neq 1$.

\itemth{3} $f$ is $g$-invariant.
\end{itemize}
Then $f(v)=0$.
\end{lem}

\bigskip

The next one is certainly well-known and might have its own interest:

\bigskip

\begin{lem}\label{lem:action tangent}
Let $V$ be a finite dimensional vector space,
let $n=\dim_\CM V$, let $G$ be a subgroup of $\GL_\CM(V)$,
let $v \in V \setminus \{0\}$ and
let $f_1$,\dots, $f_r \in \CM[V]$ be $G$-invariant and
homogeneous of respective degrees $d_1$,\dots, $d_r$. We assume that:
\begin{itemize}
\itemth{1} $\ZC(f_1,\dots,f_r)$ is a global complete intersection in $\PM(V)$.

\itemth{2} $G$ stabilizes the line $[v]$ (let $\th_v : G \to \CM^\times$ denote
the linear character defined by $g(v)=\th_v(g)v$ for all $g \in G$).

\itemth{3} $[v]$ is a smooth point of $\ZC(f_1,\dots,f_r)$.
\end{itemize}
Then the $r$-dimensional semisimple representation $\th_v^{d_1-1} \oplus \cdots \oplus \th_v^{d_r-1}$
is isomorphic to a subrepresentation $E$ of the $(n-1)$-dimensional representation $(V/[v])^*$
of $G$ and, as representations of $G$, we have an isomorphism
$$\Trm_{[v]}(\ZC(f_1,\dots,f_r)) \simeq E^\perp \otimes \th_v^{-1},$$
where $E^\perp = \{x \in V/[v]~|~\forall~\ph \in E,~\ph(x)=0\}$.
\end{lem}

\bigskip

\begin{rema}\label{eq:tangent proj}
Keep the notation of the proposition.
Recall that we have a natural identification $\Trm_{[v]}(\PM(V)) \simeq V/[v]$ but that
this identification is not $G$-equivariant. We will recall the construction of this isomorphism
in the proof of the proposition, so that we can follow the action of $G$: as we will see,
we have a natural isomorphism of $G$-representations $\Trm_{[v]}(\PM(V)) \simeq (V/[v]) \otimes \th_v^{-1}$.
Therefore, $\Trm_{[v]}(\XC)$ is naturally a subrepresentation of $(V/[v]) \otimes \th_v^{-1}$,
and this is the aim of this proposition to identify this subrepresentation.\finl
\end{rema}

\bigskip

\begin{proof}
Write $\XC=\ZC(f_1,\dots,f_r)$ and $p=[v]$.
The point $p$ of $\PM(V)$ corresponds to the homogeneous
ideal $\pG$ of $\CM[V]$ generated by $v^\perp \subset V^*$.
The local ring $\OC$ of $\PM(V)$ at $p$ is
$$\CM[V]_{(\pG)}=\{a/b~|~a, b \in \CM[V],~\text{homogeneous and}~b \not\in \pG\}.$$
We denote by $\mG$ the unique maximal ideal of $\OC$. Then
$$\mG=\{a/b~|~a,b \in \CM[V],~\text{homogeneous},~a \in \pG~\text{and}~b \not\in \pG\}.$$
The map
$$\fonction{\drm_v}{\mG}{V^*}{f}{\drm_v f}$$
induces an isomorphism
$$\d : \Trm_p(\PM(V))^*=\mG/\mG^2 \longiso v^\perp =(V/[v])^*.$$
(Recall that, if $a \in \pG$ is homogeneous of degree $m$ and $b \not\in \pG$ is homogeneous,
then $\drm_v(a/b) (v) = (\drm_v a)(v)/b(v)=0$ because
$(\drm_v a)(v)=ma(v)=0$ by Euler's identity). Note, however, that $\d$ is not $G$-equivariant
for the classical action: indeed, as $g(v)=\th_v(g)v$ for all $g \in G$, $\d$ induces
an isomorphism of $G$-modules
$$\d : \Trm_p(\PM(V))^* \longiso (V/[v])^* \otimes \th_v.$$
Through this isomorphism,
$$\Trm_p(\XC)^* \simeq \bigl((V/[v])^*/E\bigr) \otimes \th_v,$$
where
$$E= \sum_{k=1}^r \CM \d(f_k).$$
Now, since $\XC$ is smooth at $p$ and $\XC$ is a global complete intersection,
this means that the elements $\d(f_1)$,\dots, $\d(f_r)$ are linearly independent.
To conclude the proof of the proposition, it remains to notice that
$g(\d(f_k))=\th_v(g)^{d_k-1} \d(f_k)$ for all $g \in G$, which follows immediately
from the $G$-invariance of $f_k$.
\end{proof}

\bigskip

The next corollary is an immediate consequence of Lemma~\ref{lem:action tangent} in the
case where $G$ is a cyclic group.

\bigskip

\begin{coro}\label{coro:action tangent}
Let $V$ be a finite dimensional vector space,
let $n=\dim_\CM V$, let $g \in \GL_\CM(V)$, let $v \in V \setminus \{0\}$ and
let $f_1$,\dots, $f_r \in \CM[V]$ be $g$-invariant and
homogeneous of respective degrees $d_1$,\dots, $d_r$. Let $\multiset{\xi_1,\dots,\xi_n}$
be the multiset of eigenvalues of $g$. We assume that:
\begin{itemize}
\itemth{1} $\ZC(f_1,\dots,f_r)$ is a global complete intersection in $\PM(V)$.

\itemth{2} $g(v)=\xi_n v$.

\itemth{3} $[v]$ is a smooth point
of $\ZC(f_1,\dots,f_r)$.
\end{itemize}
Then $\multiset{\xi_n^{-d_1},\dots,\xi_n^{-d_r}}$ is contained in the multiset
$\multiset{\xi_n^{-1}\xi_1,\dots,\xi_n^{-1}\xi_{n-1}}$ and the list of eigenvalues
of $g$ for its action on the tangent space $\Trm_{[v]}(\ZC(f_1,\dots,f_r))$
is the multiset
$$\multiset{\xi_n^{-1}\xi_1,\dots,\xi_n^{-1}\xi_{n-1}} \setminus
\multiset{\xi_n^{-d_1},\dots,\xi_n^{-d_r}}.$$
\end{coro}

\bigskip

\begin{rema}\label{rem:lissite}
Keep the notation of Lemma~\ref{lem:action tangent} and assume moreover that $G$ is a closed reductive
subgroup of $\GL_\CM(V)$ (for instance, $G$ might be finite), that $v \in V^G$ (so that $\th_v=1$)
and that $\PM(V^G) \subset \ZC(f_1,\dots,f_r)$. We want to show that
$$\text{\it $[v]$ is a singular point of $\ZC(f_1,\dots,f_r)$.}\leqno{(\#)}$$
For this, assume that $[v]$ is smooth and let $E$ denote the $G$-stable subspace of $(V/[v])^*$
of dimension $r$ defined in the proof of Lemma~\ref{lem:action tangent} and such that
the $G$-module $\Trm_{[v]}(\ZC(f_1,\dots,f_r))$ is identified with $E^\perp$.
By hypothesis, $V^G/[v] \subset \Trm_{[v]}(\ZC(f_1,\dots,f_r))$ so, since
$G$ acts semisimply on $V$, its orthogonal $F$ in  $(V/[v])^*$ satisfies $F^G=0$.
This contradicts the fact that $E \subset F$. This shows~$(\#)$.

If $G$ is finite, $\dim_\CM V^G=1$ (resp. $\dim_\CM V^G=2$)
and $r=1$, we retrieve~\cite[Coro.~2.4]{bonnafe sarti 1} (resp. \cite[Coro.~2.9]{bonnafe sarti 1}).\finl
\end{rema}

\bigskip

\section{Set-up}\label{sec:e6}

\medskip

For the classical theory of Coxeter groups, reflection groups, Dynkin diagrams,...
we mainly refer to Bourbaki's book~\cite{bourbaki} or Brou\'e's book~\cite{broue}. For
Springer theory (and its enhancement by Lehrer-Springer), we refer to~\cite{springer},~\cite{lehrerspringer1},~\cite{lehrerspringer2}
and~\cite{lehrermichel}.

Let $V=\CM^6$ be endowed with its canonical basis $(e_1,e_2,e_3,e_4,e_5,e_6)$.
Through this basis, we identify $\GL_\CM(V)$ and $\GL_6(\CM)$ and we let it act on $V$
on the left. We denote by
$(X_1,X_2,X_3,X_4,X_5,X_6)$ the dual basis of $(e_1,e_2,e_3,e_4,e_5,e_6)$,
so that the algebra $\CM[V]$ will be identified with the polynomial algebra $\CM[X_1,X_2,X_3,X_4,X_5,X_6]$. We endow it with
the symmetric bilinear form $\langle,\rangle$ attached to the Dynkin diagram of type $\Erm_6$
(we follow the strange numbering of nodes given by Bourbaki~\cite[Chap.~6,~Planche~V]{bourbaki}):

\begin{center}
\setlength{\unitlength}{0.3mm}
\begin{picture}(320,63)
\put( 80, 40){\circle{10}}\put(77,49){$1$}
\put( 85, 40){\line(1,0){30}}
\put(120, 40){\circle{10}}\put(117,49){$3$}
\put(125, 40){\line(1,0){30}}
\put(160, 40){\circle{10}}\put(157,49){$4$}
\put(165, 40){\line(1,0){30}}
\put(200, 40){\circle{10}}\put(197,49){$5$}
\put(205, 40){\line(1,0){30}}
\put(240, 40){\circle{10}}\put(237,49){$6$}
\put(160, 35){\line(0,-1){20}}
\put(160, 10){\circle{10}}\put(145,6){$2$}
\put(-100,30){$({\mathrm{E}}_6)$}
\end{picture}
\end{center}
Recall that this means that
$$
\begin{cases}
\langle e_k, e_k \rangle = 1 & \text{if $1 \le k \le 6$,} \\
\langle e_k, e_l \rangle = -1/2 &
\text{if $1\le k \neq l \le 6$ and $\{k,l\}$ is an edge of the graph $(\Erm_6)$,} \\
\langle e_k,e_l \rangle = 0 &
\text{if $1 \le k \neq l \le 6$ and $\{k,l\}$ is not an edge of the graph $(\Erm_6)$.}
\end{cases}
$$
Recall that $\langle ,\rangle $ is non-degenerate (in fact, when restricted to $\RM^6$,
it is positive definite). For $1 \le k \le 6$, we denote by $s_k$ the orthogonal reflection
such that $s_k(e_k)=-e_k$. We set
$$W=\langle s_1,s_2,s_3,s_4,s_5,s_6 \rangle.$$
We denote by $\Orm(V)$ the orthogonal group of $V$, with respect to the bilinear
form $\langle,\rangle$. By construction, $W$ is a sugroup of $\Orm(V)$
and is called a {\it Weyl group of type $\Erm_6$}.

\bigskip

\subsection{First properties of ${\boldsymbol{W}}$}
The group $W$ is finite and acts irreducibly of $V$. Moreover,
\equat\label{eq:order}
|W|=51\,840.
\endequat
Also, the center of $W$ is trivial, so $W$ acts faithfully on $\PM(V)$.
Let $\e : W \to \mub_2=\{-1.1\}$, $w \mapsto \det(w)$. Recall that $\Ker \e$ is the
derived (i.e. commutator) subgroup of $W$, which will be denoted by $W'$. In particular,
\equat\label{eq:order dw}
|W'| = 25\,920.
\endequat
We denote by $\Deg(W)$ (resp. $\Codeg(W)$) the {\it degrees} (resp. the {\it codegrees}) of
$W$, as defined in~\cite[Chap.~4]{broue}. It turns out that
\equat\label{eq:degres}
\Deg(W)=\{2,5,6,8,9,12\} \qquad \text{and} \qquad \Codeg(W) = \{0,3,4,6,7,10\}.
\endequat
In particular, this means that there exist $6$ homogeneous $W$-invariant polynomials
$f_2$, $f_5$, $f_6$, $f_8$, $f_9$ and $f_{12}$, of respective degrees $2$, $5$, $6$, $8$, $9$ and $12$,
such that
\equat\label{eq:invariants}
\CM[V]^W = \CM[f_2,f_5,f_6,f_8,f_9,f_{12}].
\endequat
We set $\fb=(f_2,f_5,f_6,f_8,f_9,f_{12})$ and we recall that $\fb$ is not uniquely determined.
Thanks to~\eqref{eq:invariants}, we get that
$$\fonctionl{\pi_\fb}{\PM(V)}{\PM(2,5,6,8,9,12)}{[v]}{[f_2(v),f_5(v),f_6(v),f_8(v),f_9(v),f_{12}(v)]}$$
is well-defined and induces an isomorphism of varieties
\equat\label{eq:quotient}
\PM(V)/W \longiso \PM(2,5,6,8,9,12).
\endequat
The graded algebra associated with the weighted projective space $\PM(2,5,6,8,9,12)$
will be denoted by $\CM[Z_2,Z_5,Z_6,Z_8,Z_9,Z_{12}]$, with $Z_d$ being given the degree $d$ (for
all $d \in \Deg(W)$).

\bigskip

\begin{rema}\label{rem:f2}
The quadratic form $Q : V \to \CM$, $v \mapsto \langle v,v \rangle$ is $W$-invariant.
Hence $f_2$ is a scalar multiple of $Q$. Since $Q$ is positive definite when restricted
to $\RM^6$, we have in particular that $f_2(e_1) \neq 0$.\finl
\end{rema}

\bigskip

\def\Jac{{\mathrm{Jac}}}

Let $\REF(W)$ be the set of reflections of $W$ and let $\AC$ be the hyperplane arrangement
of $W$ (i.e. $\AC=\{V^s~|~s \in \REF(W)\}$). Then
\equat\label{eq:ref}
|\REF(W)|=|\AC|=\sum_{d \in \Deg(W)} (d-1) = 36.
\endequat
If $H \in \AC$, we denote by $\a_H$ an element of $V^*$ such that $H=\Ker \a_H$. We set
$$\Jac=\prod_{H \in \AC} \a_H \in \CM[V].$$
Since all the $\a_H$ are well-defined up to a non-zero
scalar, $\Jac$ is also well-defined up to a non-zero scalar.
Then
\equat\label{eq:action j}
\lexp{w}{\Jac}=\e(w)\Jac
\endequat
for all $w \in W$ and
\equat\label{eq:invariants w'}
\CM[V]^{W'}=\CM[f_2,f_5,f_6,f_8,f_9,f_{12},\Jac].
\endequat
Moreover, $\Jac$ is homogeneous of degree $36$ by~\eqref{eq:ref}.
Also, since $\Jac^2 \in \CM[V]^W$ by~\eqref{eq:action j}, there exists a unique homogeneous polynomial
$P_\fb \in \CM[Z_2,Z_5,Z_6,Z_8,Z_9,Z_{12}]$ such that
$$\Jac^2=P_\fb(f_2,f_5,f_6,f_8,f_9,f_{12}).$$
Of course, $P_\fb$ depends heavily on the choice of the family $\fb$
and, up to a non-zero scalar, on the choice of the $\a_H$'s.
It turns out that this relation generates the ideal of relations between
the functions $f_2$, $f_5$, $f_6$, $f_8$, $f_9$, $f_{12}$ and $\Jac$. In particular, the map
$$\fonctionl{\pi_\fb'}{\PM(V)}{\PM(2,5,6,8,9,12,36)}{[v]}{[f_2(v) : f_5(v) : f_6(v) : f_8(v) : f_9(v) :
f_{12}(v) : \Jac(v)]}$$
is well-defined and induces an isomorphism of varieties
\equat\label{eq:quotient w'}
\begin{array}{rcl}
\PM(V)/W' & \longiso & \{[z_2 : z_5 : z_6 : z_8 : z_9: z_{12} : j] \in \PM(2,5,6,8,9,12,36)~|~ \petitespace\\
&&  \hskip2cm j^2=P_\fb(z_2,z_5,z_6,z_8,z_9,z_{12})\}.\petitespace
\end{array}
\endequat
Finally, we denote by $\o : \PM(V)/W' \longto \PM(V)/W$ the natural morphism, which is just the quotient
map by $\mub_2 \simeq W/W'$. Through the isomorphisms~\eqref{eq:quotient} and~\eqref{eq:quotient w'},
the action of the non-trivial element of $\mub_2$ is given by the involutive automorphism $\s$
given by
$$\s [z_2 : z_5 : z_6 : z_8 : z_9 : z_{12} : j]=[z_2 : z_5 : z_6 : z_8 : z_9 : z_{12} : -j]$$
and $\o$ is given by
$$\o [z_2 : z_5 : z_6 : z_8 : z_9 : z_{12} : j] = [z_2 : z_5 : z_6 : z_8 : z_9 : z_{12}].$$

\subsection{Eigenspaces, Springer theory}
As in~\cite{bonnafe sarti 1}, an important role is played by Springer and Lehrer-Springer theory.
We recall briefly the results we will need (this subsection is a simplified version
of~\cite[\S{3.3}]{bonnafe sarti 1}, adapted to the particular case of our Weyl group $W$).
All the results stated here can be found
in~\cite{springer},~\cite{lehrerspringer1},~\cite{lehrerspringer2}.
Note that some of the proofs have been simplified in~\cite{lehrermichel}.
Let us fix now a natural number $e$. We set
$$\D(e)=\{d \in \Deg(W)~|~e\text{~divides~}d\},$$
$$\D^*(e)=\{d^* \in \Codeg(W)~|~e~\text{divides}~d^*\},$$
$$\d(e)=|\D(e)|\qquad\text{and}\qquad \d^*(e)=|\D^*(e)|.$$
With this notation, we have
\equat\label{eq:max-dim}
\d(e)=\max_{w \in W} \bigl(\dim V(w,\z_e)\bigr).
\endequat
In particular, $\z_e$ is an eigenvalue of some element of $W$
if and only if $\d(e) \neq 0$ that is, if and only if $e \in \{1,2,3,4,5,6,8,9,12\}$.
In this case, we fix an element $w_e$ of $W$ of minimal order such that
$$\dim V(w_e,\z_e)=\d(e).$$
We set for simplification $V(e)=V(w_e,\z_e)$ and $W(e)=W[V(e)]=W_{V(e)}^\setw/W_{V(e)}^\ptw$.

If $f \in \CM[V]$, we denote by $f^{[e]}$ its restriction to
$V(e)$. Note that if $d \not\in \D(e)$, then $f_d^{[e]}=0$ by
Lemma~\ref{lem:trivial}.

\bigskip

\begin{theo}[Springer, Lehrer-Springer]\label{theo:springer}
Assume that $\d(e) \neq 0$. Then:
\begin{itemize}
\itemth{a} If $w \in W$, then there exists $x \in W$ such that
$x (V(w,\z_e)) \subset V(e)$.

\itemth{b} $W(e)$ acts (faithfully) on $V(e)$
as a group generated by reflections.

\itemth{c} The family
$(f_d^{[e]})_{d \in \D(e)}$ is a family of fundamental
invariants of $W(e)$. In particular, the list of degrees
of $W(e)$ consists of the degrees of $W$ which are divisible
by $e$.

\itemth{d} We have
$$\bigcup_{w \in W} V(w,\z_e)=\bigcup_{x \in W} x(V(e)) =
\{v \in V~|~\forall~d \in \Deg(W) \setminus \D(e),~f_d(v)=0\}.$$

\itemth{e} $\d^*(e) \ge \d(e)$ with equality if and only if
$W_{V(e)}^\ptw = 1$.

\itemth{f} If $\d^*(e)=\d(e)$, then $w_e$ has order $e$,
$W(e)=W_{V(e)}^\setw=C_W(w_e)$ and the multiset of eigenvalues
(with multiplicity) of $w_e$ is
equal to $\multiset{\z_e^{1-d}}_{d \in \Deg(W)}$.
Moreover, if $w$ is such that $\dim V(w,\z_e)=\d(e)$, then
$w$ is conjugate to $w_e$.
\end{itemize}
\end{theo}

\bigskip

\begin{exemple}\label{ex:stab-cyclique}
Let $e \in \{8,9,12\}$. Then $\d(e)=\d^*(e)=1$.
So $V(e)$ is a line in $V$, and can be viewed as an
element of $\PM(V)$. It follows from the above results that
$$W_{V(e)}=\langle w_e \rangle$$
(see for instance~\cite[Rem.~3.14]{bonnafe sarti 1} for a proof).

Moreover, by Theorem~\ref{theo:springer}(f), we have $\det(w_e)=\z_e^{-36}$,
so
$$w_9, w_{12} \in W'\qquad\text{and}\qquad w_8 \not\in W'.$$
In particular, if we denote by $p_e$ (resp. $q_e$) the image of $V(e)$ in $\PM(V)/W'$
(resp. $\PM(V)/W$), then the morphism $\o$ is unramified (resp. ramified) over $q_9$ and $q_{12}$
(resp. $p_8$). This implies also that $\s(p_9) \neq p_9$ and $\s(p_{12}) \neq p_{12}$.\finl
\end{exemple}

\bigskip

\begin{exemple}\label{ex:w5}
Note that $\d^*(5)=2 > \d(5)=1$ while, if $e \in \{1,2,3,4,6,8,9,12\}$, then
$\d^*(e)=\d(e)$. If $x$ is an element of order $5$, then $x$ admits a primitive
fifth root of unity as eigenvalue, so admits $\z_5$ as an eigenvalue because
$x$ is represented by a matrix with rational coefficients. So we can take $w_5=x$,
and so $w_5$ has order $5$.

The above argument shows that the multiset of eigenvalues of $w_5$ is
$\multiset{1,1,\z_5,\z_5^2,\z_5^3,\z_5^4}$. Now, if $w$ is an element of $W$
such that $w(V(5))=V(5)$, then, since $w$ is defined over $\QM$, it stabilizes
the $\z_5^k$-eigenspace of $w$ for $k \in \{1,2,3,4\}$. Therefore, it stabilizes
the sum $E$ of these $4$ eigenspaces. But $E=\Ker(\Id_V + w_5 + w_5^2+w_5^3+w_5^4)$
is defined over $\QM$ (and so over $\RM$) so its orthogonal $E^\perp$ in $V$
with respect to the bilinear form $\langle,\rangle$ satisfies $E \oplus E^\perp=V$
(because its restriction to $\RM^6$ is positive definite). This shows that $w$ and $w_5$ stabilize
also $E^\perp$, which is necessarily equal to $\Ker(w_5-\Id_V)$. In particular
$w$ centralizes $w_5$. But, by the fourth line of Computation~\ref{comp:order 5},
we have $C_{W'}(w_5)=\langle w_5 \rangle$. So we have proved that
$$W_{V(5)}'=\langle w_5 \rangle.$$
This fact will be used in the proof of Theorem~\ref{theo:k3}.\finl
\end{exemple}

\bigskip

\section{K3 surfaces}\label{sec:k3}

\medskip

Let
$$\XC=\ZC(f_2,f_6,f_8)$$
and, for $\l$, $\mu \in \CM$, set
$$\YC_{\l,\mu}=\ZC(f_5,f_6+\l f_2^3,f_8 + \mu f_2^4).$$
Note that the invariants $f_2$ and $f_5$ are uniquely defined up to a scalar and that,
up to a scalar, every fundamental invariant of degree $6$ (resp. $8$) is of the form
$f_6 + \l f_2^3$ (resp. $f_8+\mu f_2^4 + \nu f_6f_2$) for some $\l \in \CM$
(resp. $\mu$, $\nu \in \CM$). But $\ZC(f_2,f_6+\l f_2^3,f_8+\mu f_2^4 + \nu f_6f_2)=\XC$
and $\ZC(f_5,f_6+\l f_2^3,f_8+\mu f_2^4 + \nu f_6f_2)=\YC_{\l,\mu-\l\nu}$. So this shows in particular
that $\XC$ does not depend on the choices of the fundamental invariants of $W$.
By construction, $\XC$ and $\YC_{\l,\mu}$ are $W$-stable.

\bigskip

\begin{prop}\label{prop:XY}
With the above notation, we have:
\begin{itemize}
\itemth{a} $\XC$ is a smooth irreducible surface, which is a complete intersection in $\PM(V))$.

\itemth{b} If $\l$, $\mu \in \CM$, then $\YC_{\l,\mu}$ has pure dimension $2$, and so is a complete
intersection in $\PM(V)$. If it is smooth or has only ADE singularities,
then it is irreducible.

\itemth{c} The set of $(\l,\mu) \in \CM^2$ such that $\YC_{\l,\mu}$ is smooth is a non-empty open
subset of $\CM^2$.
\end{itemize}
\end{prop}

\bigskip

%

\begin{proof}
Through the isomorphism~\eqref{eq:quotient}, we have
$$\XC/W \simeq \PM(5,9,12)\qquad\text{and}\qquad \YC_{\l,\mu}/W \simeq \PM(2,9,12),$$
so $\XC/W$ and $\YC_{\l,\mu}/W$ are irreducible of dimension $2$. Since $W$ is finite,
this implies that $\XC$ and $\YC_{\l,\mu}$ are of pure dimension $2$ and so are complete
intersections. In particular, they are connected~\cite[Chap.~II,~Exer.~8.4(c)]{hartshorne}.

\medskip

(a) By Computation~\ref{comp:1}, the open affine chart of $\XC$ defined by $x_6 \neq 0$
is smooth. This shows that the singular locus $\SC$ of $\XC$ is contained in the projective
hyperplane $\PM(H)$, where $H=\{(x_1,x_2,x_3,x_4,x_5,x_6) \in \CM^6~|~x_6=0\}$.
Since $W$ acts on $\XC$, we have that $\SC$ is contained in $\PM(\cap_{w \in W} w(H))$.
But $\cap_{w \in W} w(H)$ is a $W$-stable proper subspace of $V$, so it is equal to $\{0\}$ since
$W$ acts irreducibly on $V$. Hence $\SC=\vide$.
This shows that $\XC$ is smooth and, in particular, irreducible (because it is connected).

\medskip

(b) We already know that $\YC_{\l,\mu}$ is connected. If moreover it admits only
ADE singularities, then it is necessarily irreducible.

\medskip

(c) Let $U$ be the set of $(\l,\mu) \in \CM^2$ such that $\YC_{\l,\mu}$ is smooth.
It is clear that $U$ is open, so we only need to prove that it is non-empty.
With the particular choice of fundamental invariants given in Appendix~\ref{appendix:magma},
we have that the open affine chart of $\YC_{0,0}$ defined by $x_6 \neq 0$ is smooth
by Computation~\ref{comp:1}. The same argument as in~(a) allows to conclude that $\YC_{0,0}$
is smooth and irreducible.
\end{proof}

\bigskip

We are now ready to state the first main result of this paper:

\bigskip

\begin{theo}\label{theo:k3}
Let $\l$, $\mu \in \CM$ be such that $\YC_{\l,\mu}$ admits only ADE singularities and let $\VC$
be one of the varieties $\XC$ or $\YC_{\l,\mu}$. Then $\VC/W'$ is a K3 surface with only ADE singularities
and its minimal smooth resolution $\VCt$ is a smooth K3 surface.
\end{theo}

\bigskip

\begin{proof}
The proof will be given in the next subsections, following the same lines as~\cite[Theo.~5.4]{bonnafe sarti 1}.
More precisely, if $\VC$ denotes one of the above varieties $\XC$ or $\YC_{\l,\mu}$,
we will prove in the next subsections the following three facts:
\begin{itemize}
 \itemth{A} The smooth locus of $\VC/W'$ admits a symplectic form (see Lemma~\ref{lem:symplectique});

 \itemth{B} The variety $\VC/W'$ has only ADE singularities (see Lemma~\ref{lem:ADE});

 \itemth{C} The Euler characteristic of $\VC/W'$ is positive (see Lemma~\ref{lem:euler}).
\end{itemize}
Therefore, by~(A) and~(B), the variety $\VC/W'$ is a symplectic singularity and its minimal smooth resolution
$\VCt$ is a crepant resolution, i.e. a symplectic resolution. So $\VCt$ admits a symplectic form.
By the classification of surfaces, this forces $\VCt$ to be a smooth K3 surface or an abelian
variety. But, by~(B), the Euler characteristic of $\VCt$ is greater than or equal
to the one of $\VC/W'$, so the Euler characteristic of $\VCt$ is positive by~(C).
Since an abelian variety has Euler characteristic $0$, this shows that
$\VCt$ is a smooth K3 surface.
\end{proof}

\bigskip

So it remains to prove the three facts~(A),~(B) and~(C) used in the above proof.

\medskip

\boitegrise{\vphantom{$\frac{A}{a}$}{\bf Notation.} For the rest of this section, we fix $\l$, $\mu$ in $\CM$
such that $\YC_{\l,\mu}$ admits only ADE singularities and we denote by $\VC$ a variety
which can be $\XC$ or $\YC_{\l,\mu}$.}{0.75\textwidth}

\smallskip

Recall that this implies that $\VC$ is irreducible and normal and, in particular, that $\VC/W'$ is
irreducible and normal too.

\medskip

\subsection{Symplectic form}
Through the isomorphism~\eqref{eq:quotient w'}, we have
$$\XC/W' \simeq \{[z_5 : z_9 : z_{12} : j] \in \PM(5,9,12,36)~|~j^2=P_\fb(0,z_5,0,0,z_9,z_{12})$$
$$\YC_{\l,\mu}/W' \simeq \{[z_2 : z_9 : z_{12} : j] \in
\PM(2,9,12,36)~|~j^2=P_\fb(z_2,0,-\l z_2^3,-\mu z_2^4,z_9,z_{12})\}.
\leqno{\text{and}}$$
But note that $\PM(5,9,12,36)=\PM(5,3,4,12)$, so there exists a unique homogeneous polynomial $Q_\fb \in \CM[Y_5,Y_3,Y_4]$
of degree $24$ (where $Y_i$ is endowed with the degree $i$) such that
$P_\fb(0,z_5,0,0,z_9,z_{12})=Q_\fb(z_5^3,z_9,z_{12})$. Hence,
\equat\label{eq:xw'}
\XC/W' \simeq \{[y_5 : y_3 : y_4 : j] \in \PM(5,3,4,12)~|~j^2=Q_\fb(y_5,y_3,y_4)\}.
\endequat
So the degree of the equation defining $\XC/W'$ (namely, $24$) is equal to the sum of
the weights of the projective space (namely, $5+3+4+12$).
By~\cite[Lem.~A.1]{bonnafe sarti 1}, this implies that the smooth locus of $\XC/W'$ is endowed with
a symplectic form.

On the other hand, $\PM(2,9,12,36)=\PM(1,9,6,18)=\PM(1,3,2,6)$ so there exists a unique homogeneous
polynomial $Q_\fb^{\l,\mu} \in \CM[Y_1,Y_3,Y_2]$ of degree $12$
(where $Y_i$ is endowed with the degree $i$) such that
$P_\fb(z_2,0,-\l z_2^3,-\mu z_2^4,z_9,z_{12})=Q_\fb^{\l,\mu}(z_2^3,z_9,z_{12})$. Hence,
\equat\label{eq:yw'}
\YC_{\l,\mu}/W' \simeq \{[y_1 : y_3 : y_2 : j] \in \PM(1,3,2,6)~|~j^2=Q_\fb^{\l,\mu}(y_1,y_3,y_2)\}.
\endequat
So the degree of the equation defining $\YC_{\l,\mu}/W'$ (namely, $12$) is equal to the sum of
the weights of the projective space (namely, $1+3+2+6$).
By~\cite[Lem.~5.4]{bonnafe sarti 1}, this implies that the smooth locus of $\XC/W'$ is endowed with
a symplectic form. Therefore, we have proved the following lemma, which corresponds to
the Fact~(A) stated in the proof of Theorem~\ref{theo:k3}:

\medskip

\begin{lem}\label{lem:symplectique}
The smooth locus of $\VC/W'$ admits a symplectic form.
\end{lem}

\medskip

\begin{rema}\label{rem:branch}
In both cases, the ramification locus of the morphism $\o : \VC/W' \to \VC/W$ is given
by the equation $j=0$ (in the models given by equations~\eqref{eq:xw'} or~\eqref{eq:yw'}).\finl
\end{rema}

\bigskip

\subsection{Singularities}
We aim to prove here the following lemma, which corresponds to
the Fact~(B) stated in the proof of Theorem~\ref{theo:k3}:

\bigskip

\begin{lem}\label{lem:ADE}
The surface $\VC/W'$ has only ADE singularities.
\end{lem}

\bigskip

\begin{proof}
First, as $\VC$ has only ADE singularities and $W'$ has index $2$ in $W$,
every point of $\VC/W'$ lying above a smooth point of $\VC/W$ is smooth or
is an ADE singularity by~\cite[Coro.~B.7]{bonnafe sarti 1}. So it remains only to study the points
lying above the singular points of $\VC/W$.

The singular points of $\XC/W \simeq \PM(5,3,4)$ are $q_5=[1 : 0 : 0]$, $q_9=[0 : 1 : 0]$ and $q_{12}=[0 : 0 : 1]$.
The singular points of $\YC_{\l,\mu}/W \simeq \PM(1,3,2)$ are $q_9=[0 : 1 : 0]$ and $q_{12}=[0 : 0 : 1]$.
Note that the notation $q_9$ and $q_{12}$ is consistent with Example~\ref{ex:stab-cyclique},
as they correspond to the points $q_9$ and $q_{12}$ defined in this example through the
embeddings $\VC/W \injto \PM(V)/W \simeq \PM(2,5,6,8,9,12)$. Still by Example~\ref{ex:stab-cyclique},
the morphism $\VC/W' \to \VC/W$ is unramified above $q_9$ and $q_{12}$. Therefore, the points $p_9$
and $p_9'=\s(p_9)$ (resp. $p_{12}$ and $p_{12}'=\s(p_{12})$) of $\VC/W'$ are distinct
and have the same type of singularities
than the point $q_9$ (resp. $q_{12}$) of $\VC/W$. But $q_9$ is a singular point of type $A_2$
of $\PM(5,3,4)$ or $\PM(1,3,2)$ and $q_{12}$ is a singular point of type $A_3$ (resp. $A_1$)
of $\PM(5,3,4)$ (resp. $\PM(1,3,2)$). This shows that the following results holds:

\medskip

\begin{quotation}
\begin{lem}\label{lem:p9 p12}
We have:
\begin{itemize}
\itemth{a} The points $p_9$ and $p_9'$ are $A_2$ singularities of $\XC/W'$ and the points
$p_{12}$ and $p_{12}'$ are $A_3$ singularities of $\XC/W'$.

\itemth{b} The points $p_9$ and $p_9'$ are $A_2$ singularities of $\YC_{\l,\mu}/W'$ and the points
$p_{12}$ and $p_{12}'$ are $A_1$ singularities of $\YC_{\l,\mu}/W'$.
\end{itemize}
\end{lem}
\end{quotation}

\medskip

Therefore, it remains to prove that the points of $\XC/W'$ lying above $q_5$ are ADE singularities.
For this, note that
$$\pi^{-1}(q_5)=\{x \in \PM(V)~|~f_2(x)=f_6(x)=f_8(x)=f_9(x)=f_{12}(x)=0\}.$$
But, by~\eqref{eq:max-dim} and Theorem~\ref{theo:springer}(d) applied to the case where $e=5$,
we get that $V(5)$ is a line
in $V$, so may be viewed as a point of $\PM(V)$ and $\pi_\fb^{-1}(q_5)$ is the $W$-orbit
of $V(5)$. Now, $\d^*(5)=2 > \d(5)=1$, so it follows from Theorem~\ref{theo:springer}(f) that
$W_{V(5)}^\ptw \neq 1$. By Steinberg's Theorem~(see for instance~\cite[Theo.~4.7]{broue}),
this shows that $W_{V(5)}^\ptw$
contains a reflection, and so the stabilizer $G$ of $V(5)$ in $W$ contains
a reflection. In particular, $G$ is not contained in $W'$. This proves that
the morphism $\XC/W' \to \XC/W$ is ramified above $q_5$: we denote by $p_5$
the unique point of $\XC/W'$ lying above $q_5$.

Now, Example~\ref{ex:w5} shows that the stabilizer
of $V(5)$ in $W'$ is $\langle w_5 \rangle$. So, in order to determine the type of
singularity of $\XC/W'$ at $p_5$, we only need to determine the two eigenvalues of
$w_5$ for its action on $\Trm_{V(5)}(\XC)$. This is easily done thanks to
Corollary~\ref{coro:action tangent}: the two eigenvalues are $\z_5$ and $\z_5^{-1}$.
We have thus proved the following result:

\medskip

\begin{quotation}
\begin{lem}\label{lem:p5}
The point $p_5$ is an $A_4$ singularity of $\XC/W'$.
\end{lem}
\end{quotation}

\medskip

This completes the proof of Lemma~\ref{lem:ADE}.
\end{proof}

\medskip

\subsection{Euler characteristic}
Since $\VC$ is a complete intersection which is smooth or has only ADE singularities,
its cohomology (with coefficients in $\CM$) is concentrated in even degree~\cite[Theo.~2.1,~Lem.~3.2~and~Example~3.3]{dimca}. Now, $\Hrm^k(\VC/W',\CM) \simeq \Hrm^k(\VC,\CM)^{W'}$,
so the cohomology of $\VC/W'$ is concentrated in even degree. This implies the next lemma, which
corresponds to the Fact~(C) stated in the proof of Theorem~\ref{theo:k3}, and completes
the proof of Theorem~\ref{theo:k3}:

\bigskip

\begin{lem}\label{lem:euler}
The Euler characteristic of $\VC/W'$ is positive.
\end{lem}

\bigskip

\boitegrise{\vphantom{$\frac{A}{a}$}{\bf Notation.}
{\it The minimal smooth resolution of $\XC/W'$ will be denoted by $\r : \XCt \to \XC/W'$.
Theorem~\ref{theo:k3} says $\XCt$ is a smooth projective K3 surface.}}{0.75\textwidth}

\bigskip

\section{Some numerical data for the surface ${\boldsymbol{\XC/W'}}$}\label{sec:xw}

\medskip

We complete here the qualitative result given by Theorem~\ref{theo:k3} with
some concrete results concerning the surface $\XC/W'$
(type of singularities, equation, coordinates of singular points, cohomology,...).
These informations will be used in the next section to obtained further properties
of the K3 surface $\XCt$ (Picard lattice, elliptic fibration,...).

\bigskip

\subsection{Singularities}
In the course of the proof of Theorem~\ref{theo:k3}, we have obtained some quantitative
results (see Lemmas~\ref{lem:p9 p12} and~\ref{lem:p5}). We complete
them by determining all the singularities
of $\XC/W'$:

\bigskip

\begin{prop}\label{prop:singularities X}
The surface $\XC/W'$ admits $A_4 + 2\, A_3 + 3\, A_2 + 2\, A_1$ singularities.
\end{prop}

\bigskip

\begin{proof}
For proving this proposition, we must investigate fixed points under various elements of $W'$,
up to conjugacy. Note also that we have already found $A_4 + 2\, A_3 + 2\, A_2$ singularities
in $\XC/W'$
(see Lemma~\ref{lem:p9 p12} and~\ref{lem:p5}), given by the points $p_5$, $p_9$, $p_9'$,
$p_{12}$ and $p_{12}'$ lying respectively above the points $q_5=[1 : 0 : 0]$, $q_9=[0 : 1 : 0]$
and $q_{12}=[0 : 0 : 1]$ of $\XC/W \simeq \PM(5,3,4)$.

Note also that, for any $w \in W$, the fixed point scheme $\XC^w$ is also smooth and in particular
is reduced (so, in the {\sc Magma} computations necessary in the proof of the proposition,
we do not need to compute its reduced subscheme for determining the exact number of points).
So let us now investigate the fixed points subscheme, by letting the order of $w$ increasing.

\medskip

\begin{quotation}
\begin{lem}[Elements of order 2]\label{lem:order 2}
There are two conjugacy classes of elements of order $2$: the one of the element $w_2$ defined
after Theorem~\ref{theo:springer} with eigenvalues $\multiset{1,1,-1,-1,-1,-1}$
and the one of an element $v_2$ admitting $\multiset{1,1,1,1,-1,-1}$ as a list
of eigenvalues. Moreover:
\begin{itemize}
\itemth{a} The image of $\XC^{w_2}$ in $\XC/W'$ is contained in $\{p_{12},p_{12}'\}$.

\itemth{b} The image of $\XC^{v_2}$ in $\XC/W'$ consists in two elements $p_1$ and $p_1'$
which are both $A_1$ singularities.
\end{itemize}
\end{lem}

\bigskip

\begin{proof}[Proof of Lemma~\ref{lem:order 2}]
The first statement is proved in Computation~\ref{comp:order 2}.

\medskip

(a) By Computation~\ref{comp:w2}, $\XC^{w_2}$ has dimension $0$ and is contained in $\ZC(f_5,f_9)$,
so its image in $\XC/W$ is $q_{12}$. This proves~(a).

\medskip

(b) By the first four commands in Computation~\ref{comp:v2}, we know that
$\XC^{v_2}$ has dimension $0$ and consists of $96$ points. The fifth command shows that,
if $g \in W' \setminus \{1,v_2\}$, then
$$\XC^g \cap \XC^{v_2} = \vide.\leqno{(\clubsuit)}$$
If the reader wants to check this computation,
he or she must be aware that it takes about 10 minutes on a standard computer.
We now
need to understand how many $W'$-orbits meet $\XC^{v_2}$. For this, let
$w \in W'$ and $x \in \XC^{v_2}$ be such that $w(x) \in \XC^{v_2}$. Then
$w^{-1}v_2w(x)=x$ and so it follows from $(\clubsuit)$ that $w^{-1}v_2w=v_2$ or,
in other words, that $w \in C_{W'}(v_2)$.

But $|C_{W'}(v_2)|=96$ thanks to the last two commands of Computation~\ref{comp:v2}.
Since $C_{W'}(v_2)/\langle v_2 \rangle$ acts freely on $\XC^{v_2}$ by~$(\clubsuit)$, there are
two $C_{W'}(v_2)$-orbits in $\XC^{v_2}$, and we denote by $p_1$ and $p_1'$
their image in $\XC/W'$. Still by~$(\clubsuit)$, they are both singularities of type $A_1$
(indeed, if $x \in \XC^{v_2}$, the action of $v_2$ on $\Trm_x(\XC)$ has no eigenvalue
equal to $1$, because $x$ is isolated; thus, $v_2$ acts as $-\Id_{\Trm_x(\XC)}$).
\end{proof}
\end{quotation}

\bigskip

\begin{quotation}
\begin{lem}[Elements of order 3]\label{lem:order 3}
There are four conjugacy classes of elements of order $3$: the one of the element $w_3$ defined
after Theorem~\ref{theo:springer} with eigenvalues $\multiset{\z_3,\z_3,\z_3,\z_3^2,\z_3^2,\z_3^2}$,
the one of its inverse $w_3^{-1}$, the one of an element $v_3$ with eigenvalues $\multiset{1,1,\z_3,\z_3,\z_3^2,\z_3^2}$
and the one of an element $u_3$ with eigenvalues $\multiset{1,1,1,1,\z_3,\z_3^2}$. Moreover:
\begin{itemize}
\itemth{a} $\XC^{w_3}=\XC^{w_3^{-1}}$ has pure dimension $1$.

\itemth{b} $\XC^{v_3}$ contains $12$ points. Moreover, if $x \in \XC^{v_3}$, then $W_x'$ has order
$9$ and acts as a reflection group on $\Trm_x(\XC)$. Hence, the image of $x$ in $\XC/W'$ is smooth.

\itemth{c} $\XC^{u_3}$ contains $96$ points. Moreover, if $x \in \XC^{u_3}$, then one of the following holds:
\begin{itemize}
\itemth{c1} $W_x'=\langle u_3 \rangle$ and $u_3$ acts on $\Trm_x(\XC)$ with eigenvalues $\z_3$ and $\z_3^2$.
Hence the image of $x$ in $\XC/W'$ is an $A_2$ singularity.

\itemth{c2} $W_x'$ has order $9$ and acts as a reflection group on $\Trm_x(\XC)$.
Hence, the image of $x$ in $\XC/W'$ is smooth.
\end{itemize}
Also, all the points $x$ satisfying~{\rm (c1)} belong to the same $W'$-orbit, and their image
$p_2$ in $\XC/W'$ is not equal to $p_9$ or $p_9'$.
\end{itemize}
\end{lem}

\bigskip

\begin{proof}[Proof of Lemma~\ref{lem:order 3}]
The first statement and~(a) are both proved in Computation~\ref{comp:order 3}.

\medskip

(b) The first three commands of Computation~\ref{comp:v3} show that $\XC^{v_3}$ has dimension $0$.
The next two commands of Computation~\ref{comp:v3} show
that the set $G(v_3)$ of elements $w \in W'$ such that $\XC^{v_3} \cap \XC^w \neq \vide$
has cardinality $9$. The next three commands of Computation~\ref{comp:v3} show that $G(v_3)$ is a group
isomorphic to $\mub_3 \times \mub_3$. The last command of Computation~\ref{comp:v3} show that
$\XC^{v_3} \subset \XC^w$ for all $w \in G(v_3)$.

This implies that, if $x \in \XC^{v_3}$, then its stabilizer in $W'$ is isomorphic to
$\mub_3 \times \mub_3$. But, up to conjugacy, the only subgroup of $\GL_\CM(\Trm_x(\XC)) \simeq \GL_2(\CM)$
isomorphic to $\mub_3 \times \mub_3$ is $\langle \diag(\z_3,1),\diag(1,\z_3) \rangle$,
which is a reflection group.

\medskip

(c) The first four commands of Computation~\ref{comp:u3} show that $\XC^{u_3}$ has dimension $0$
and contains $96$ points. The next two commands of Computation~\ref{comp:u3} show that
the normalizer $N(u_3)=N_{W'}(\langle u_3\rangle)$ has order $216$. Note that $N(u_3)$ acts
on $\XC^{u_3}$. Also, the scheme $\XC^{u_3}$ is naturally defined over $\QM$
(we denote by $\XC_\QM^{u_3}$ its $\QM$-form) and computations
in {\sc Magma} are performed over $\QM$. Therefore, the last three commands of
Computation~\ref{comp:u3} show that the scheme $\XC_\QM^{u_3}$ has three
irreducible components $\IC_1$, $\IC_2$ and $\IC_3$, admitting respectively $72$, $12$ and $12$ points over $\CM$:
we set $A=\IC_1(\CM)$ and $B=\IC_2(\CM) \cup \IC_3(\CM)$. In particular,
$|A|=72$ and $|B|=24$. Note that the group $N(u_3)/\langle u_3 \rangle$ acts on $A$ and $B$.

\medskip

(c1) The first three lines of Computation~\ref{comp:A} show that, if $w \in W' \setminus \langle u_3 \rangle$,
then $A \cap \XC^w = \vide$. This shows that $N(u_3)/\langle u_3 \rangle$
acts freely on $A$. Since $|A|=72=|N(u_3)/\langle u_3 \rangle|$, this also shows that
$N(u_3)$ acts transitively: we denote by $p_3'$ the unique point in the image of $A$ in $\XC/W'$.
Note that $p_2 \not\in\{p_9,p_9'\}$ by the last line of Computation~\ref{comp:A}.

Since the stabilizer of any point $x$ in $A$
is equal to $\langle u_3 \rangle$, we need to determine the eigenvalues of $u_3$
for its action on $\Trm_x(\XC)$.
Now, let $v \in V \setminus \{0\}$ be such that $[v] \in A$. Then $u_3(v)=\xi v$ for some
$\xi \in \mub_3$. Since $x$ is a smooth point of $\XC$, it follows from
Corollary~\ref{coro:action tangent} that the multiset $\multiset{\xi^{-2},\xi^{-6},\xi^{-8}}$
is contained in the multiset $\multiset{\xi^{-1},\xi^{-1},\xi^{-1},\xi^{-1},\xi^{-1}\z_3,\xi^{-1}\z_3^2}$.
This forces $\xi=1$ and, still by Corollary~\ref{coro:action tangent},
this implies that the eigenvalues of $u_3$ for its action on $\Trm_x(\XC)$
are $\z_3$ and $\z_3^2$. This shows that $p_2$ is an $A_2$ singularity.

\medskip

(c2) The first three lines of Computation~\ref{comp:B} show that the set $T_B$
of elements $w$ of $W'$ such that $B \cap \XC^w \neq \vide$ contains $15$ elements.
By the next three lines of Computation~\ref{comp:B}, the group $G_B$ generated by $T_B$
has order $27$ and is contained un $N(u_3)$. Now, let $x \in B$.
Then its stabilizer $W_x'$ is contained in $G_B$. Since $|N(u_3)|=216$ and $|B|=24$,
this forces $|W_x'| \ge 9$. But $W_x' \neq G_B$ because there are elements $w \in G_B$ such that
$B \cap \XC^w = \vide$ (see the seventh line of Computation~\ref{comp:B}).
So $|W_x'|=9$. Finally, by the last command of Computation~\ref{comp:B},
we have that $G_B \simeq \mub_3 \times \mub_3 \times \mub_3$, so $W_x' \simeq \mub_3 \times \mub_3$.
Now, as in~(b), we conclude that the image of $x$ in $\XC/W'$ is smooth because
the unique subgroup of $\GL_2(\CM)$ isomorphic to $\mub_3 \times \mub_3$
is generated by reflections.
\end{proof}
\end{quotation}

\bigskip

\begin{quotation}
\begin{lem}[Elements of order 4]\label{lem:order 4}
There are two conjugacy classes of elements of order $4$: the one of the element $w_4$ defined
after Theorem~\ref{theo:springer} with eigenvalues $\multiset{1,1,i,i,-i,-i}$
and the one of an element $v_4$ admitting $\multiset{1,1,-1,-1,i,-i}$ as a list
of eigenvalues. Moreover:
\begin{itemize}
\itemth{a} The image of $\XC^{w_4}$ in $\XC/W'$ is $\{p_{12},p_{12}'\}$.

\itemth{b} $\XC^{v_4}=\vide$.
\end{itemize}
\end{lem}

\bigskip

\begin{proof}[Proof of Lemma~\ref{lem:order 4}]
This follows immediately from Computation~\ref{comp:w4 v4}.
\end{proof}
\end{quotation}

\bigskip

\begin{quotation}
\begin{lem}[Elements of order 5]\label{lem:order 5}
There is a unique conjugacy class of elements of order $5$, namely the one of the element $w_5$ defined
after Theorem~\ref{theo:springer},
which has $\multiset{1,1,\z_5,\z_5^2,\z_5^3,\z_5^4}$ as list of eigenvalues.
Moreover, $\XC^{w_5}$ has dimension $0$ and its image in $\XC/W'$ is the point
$p_5$ defined in the proof of Lemma~\ref{lem:ADE} (and recall that it is
an $A_4$ singularity by Lemma~\ref{lem:p5}).
\end{lem}

\bigskip

\begin{proof}[Proof of Lemma~\ref{lem:order 5}]
The first statement is proved in the first two lines of Computation~\ref{comp:order 5}.
The last three lines of Computation~\ref{comp:order 5}
show that $\XC^{w_5}$ has dimension $0$ and contains $4$ points.

Now, for $1 \le k \le 4$, let $v_k \in V \setminus\{0\}$ be such that
$w_5(v_k)=\z_5^k v_k$. Then $f_2(v_k)=f_6(v_k)=f_8(v_k)=f_9(v_k)=f_{12}(v_k)=0$
by Lemma~\ref{lem:trivial}, so $[v_k] \in \XC^{w_5}$. Therefore,
$$\XC^{w_5}=\{[v_1],[v_2],[v_3],[v_4]\}$$
and the last statement now follows from the proof of Lemma~\ref{lem:ADE}.
\end{proof}
\end{quotation}

\bigskip

\begin{quotation}
\begin{lem}[Elements of order 6]\label{lem:order 6}
There are six conjugacy classes of elements of order $6$: the one of the element $w_6$ defined
after Theorem~\ref{theo:springer} with eigenvalues $\multiset{\z_6,\z_6,\z_6^{-1},\z_6^{-1},\z_3,\z_3^{-1}}$,
the one of $w_6^{-1}$, and four others whose elements $w$ satisfies $\dim \Ker(w-\z_6 \Id_V) \le 1$.
Moreover:
\begin{itemize}
\itemth{a} The image of $\XC^{w_6}=\XC^{w_6^{-1}}$ in $\XC/W'$ is $\{p_{12},p_{12}'\}$.

\itemth{b} If $w$ is an element of order $6$ such that $\dim \Ker(w-\z_6 \Id_V) \le 1$, then $\XC^w=\vide$.
\end{itemize}
\end{lem}

\bigskip

\begin{proof}[Proof of Lemma~\ref{lem:order 6}]
The first statement follows from the first 8 commands of Computation~\ref{comp:order 6}.
(a) follows from the next two commands of Computation~\ref{comp:order 6}. (b)
follows from the last command of Computation~\ref{comp:order 6}.
\end{proof}
\end{quotation}

\bigskip

\begin{quotation}
\begin{lem}[Elements of order 9 or 12]\label{lem:order 9 12}
There are two conjugacy classes of elements of order $9$ (resp. $12$): the one of the element $w_9$
(resp. $w_{12}$) defined
after Theorem~\ref{theo:springer} and the one of $w_9^{-1}$ (resp. $w_{12}^{-1}$).
\begin{itemize}
\itemth{a} The image of $\XC^{w_9}=\XC^{w_9^{-1}}$ is equal to $\{p_9,p_9'\}$.

\itemth{b} The image of $\XC^{w_{12}}=\XC^{w_{12}^{-1}}$ is equal to $\{p_{12},p_{12}'\}$.
\end{itemize}
\end{lem}

\bigskip

\begin{proof}[Proof of Lemma~\ref{lem:order 9 12}]
This is proved in Computation~\ref{comp:order 9 12}, following the same arguments
as in the proof of Lemma~\ref{lem:order 5}.
\end{proof}
\end{quotation}

\bigskip

Since the order of an element of $W'$ belongs to $\{1,2,3,4,5,6,9,12\}$
by Computation~\ref{comp:order}, we have investigated all the possibilities.
So the proposition follows from the series of Lemmas~\ref{lem:order 2}-\ref{lem:order 9 12}.
\end{proof}

\bigskip

\begin{rema}\label{rem:w3}
The fact that $w_3$ and $w_3^{-1}$ are not conjugate in $W'$ can also be checked without
computer: in fact, they are not conjugate in $\Srm\Orm(V)$ by~\cite[Lemma~1.7]{bonnafe g32}.\finl
\end{rema}

\bigskip

For the singular points of $\XC/W'$,
we keep the notation of the proof of Proposition~\ref{prop:singularities X}:
\begin{itemize}
\item[$\bullet$] The two singular points of type $A_1$ are denoted by $p_1$ and $p_1'$.
Note that they are both $\s$-fixed because they lie above smooth
points of $\XC/W$ (so they must lie on the ramification locus of $\o : \XC/W' \to \XC/W$).

\item[$\bullet$] The singular point of type $A_2$ coming from the fixed point
locus $\XC^{u_3}$, where $u_3$ has order $3$ and satisfies $\dim V^{u_3}=4$,
is denoted by $p_2$. Again, $\s(p_2)=p_2$.

\item[$\bullet$] The singular points $p_5$, $p_9$, $p_9'$, $p_{12}$
and $p_{12}'$ involved in Lemma~\ref{lem:ADE} are of type $A_4$, $A_2$, $A_2$, $A_3$ and $A_3$
respectively.
\end{itemize}

\bigskip

\subsection{Two smooth rational curves}
Recall that $V(3)$ is the eigenspace of $w_3$ for the eigenvalue $\z_3$ and denote by
$V^-(3)$ the eigenspace of $w_3$ for the eigenvalue $\z_3^{-1}$. By~\eqref{eq:max-dim},
they both have dimension $3$ and so
\equat\label{eq:V3}
V=V(3) \oplus V^-(3).
\endequat
Let $\CC^+=\PM(V(3)) \cap \XC$ and $\CC^-=\PM(V^-(3)) \cap \XC$.
By Theorem~\ref{theo:springer}(f),
$w_3$ and $w_3^{-1}$ are conjugate in $W$, but it follows from Lemma~\ref{lem:order 3} that
they are not conjugate in $W'$. Fix $g \in W$ be such that $w_3^{-1}=gw_3g^{-1}$. Then
$g \not\in W'$ and $g(V(3))=V^-(3)$. This shows that
\equat\label{eq:tau c3}
\CC^-=g(\CC^+).
\endequat

\bigskip

\begin{lem}\label{lem:courbe 10}
The schemes $\CC^+$ and $\CC^-$ are smooth irreducible curves of genus $10$.
\end{lem}

\bigskip

\begin{proof}
By~\eqref{eq:V3}, $\PM(V)^{w_3} =\PM(V(3)) \dotcup \PM(V^-(3))$, where $\dot{\cup}$ means
a disjoint union. So
\equat\label{eq:xw3}
\XC^{w_3} = \CC^+ \dotcup \CC^-.
\endequat
Since $\XC^{w_3}$ is smooth, this implies that $\CC^+$ and $\CC^-$ are smooth.
Now, by Lemma~\ref{lem:trivial}, the restriction of $f_2$, $f_5$ and $f_8$ to $V(3)$
are equal to $0$. Therefore,
$$\CC^+=\{x \in \PM(V(3))~|~f_6(v)=0\}$$
This shows that $\CC^+$ is a connected curve of $\PM(V(3))$. Since it is smooth,
it must be irreducible. Moreover, it is of degree $6$, so
its genus is equal to $10$.
\end{proof}

\bigskip

Let $\CC_5^+$ (resp. $\CC_5^-$) denote the image of $\CC^+$ (resp. $\CC^-$) in $\XC/W'$
and let $\CC_5$ denote the image of $\CC^+$ in $\XC/W \simeq\PM(5,3,4)$. Note that
$\CC_5$ is also the image of $\CC^-$ by~\eqref{eq:tau c3}. Moreover,
\equat\label{eq:c5}
\CC_5=\{[y_5,y_3,y_4] \in \PM(5,3,4)~|~y_5=0\}=\PM(3,4)=\PM^1(\CM).
\endequat
Indeed, $\CC_5$ is irreducible, of dimension $1$ and contained in
$\{[y_5 : y_3 : y_4] \in \PM(5,3,4)~|~y_5=0\}$ by Lemma~\ref{lem:trivial}
(see also Theorem~\ref{theo:springer}(d)).

\bigskip

\begin{prop}\label{prop:c5}
We have $\CC_5^-=\s(\CC_5^+) \neq \CC_5^+$. Moreover, $\CC_5^+$ and $\CC_5^-$ are both isomorphic
to $\PM^1(\CM)$, intersect transversely along only one point and satisfy
$$\CC_5^+ \cup \CC_5^- = \{[y_5 : y_3 : y_4 : j] \in \XC/W'~|~y_5=0\}.$$
\end{prop}

\bigskip

\begin{proof}
The fact that $\CC_5^-=\s(\CC_5^+)$ follows from~\eqref{eq:tau c3}.
Now, the irreducibility of $\CC^+$ and $\CC^-$ implies the irreducibility
of $\CC_5^+$ and $\CC_5^-$. Also, by~\eqref{eq:c5},
$$\CC_5^+ \cup \CC_5^- \simeq  \{[y_3 : y_4 : j] \in \PM(3,4,12)~|~j^2=Q_\fb(0,y_3,y_4)\}.$$
But $\PM(3,4,12)=\PM(3,1,3)=\PM(1,1,1)=\PM^2(\CM)$, so there exists a polynomial
$Q_3 \in \CM[T_3,T_4]$, which is homogenous of degree $2$ (with $T_k$ of degree $1$) and
such that $Q_\fb(0,y_3,y_4)=Q_3(y_3^4,y_4^3)$. Therefore,
$$\CC_5^+ \cup \CC_5^- \simeq  \{[t_3 : t_4 : j] \in \PM^2(\CM)~|~j^2=Q_3(t_3,t_4)\}.$$
It just remains to prove that the polynomial $Q_3$ is the square of a linear form:
in other words, we only need to prove that $\CC_5^+ \neq \CC_5^-$.

Since $\CC^+$ is irreducible of degree $6$, it cannot be contained in a union of projective
lines of $\PM(V(3))$. Since moreover $\d^*(3)=\d(3)=3$, Theorem~\ref{theo:springer}(f)
thus implies that there exists $v \in V(3) \setminus \{0\}$ such that
$[v] \in \CC^+$ and $W_v=1$. We now only need to prove that $\pi_\fb'([v]) \not\in \CC^-$.
So assume that $\pi_\fb'([v]) \in \CC^-$. This would imply that there
exists $w \in W'$ such that $w(v) \in V^-(3)$. Consequently,
$$w^{-1}w_3ww_3(v)=v,$$
and so $w^{-1}w_3ww_3=1$ since $W_v=1$. Hence, $w_3$ and $w_3^{-1}$ are conjugate in $W'$,
which is impossible by Lemma~\ref{lem:order 3}.
\end{proof}

\bigskip

By exchanging $p_9$ and $p_9'$ (resp. $p_{12}$ and $p_{12}'$) if necessary,
we may assume that
$$p_9, p_{12} \in \CC_5^+\qquad\text{and}\qquad p_9',p_{12}' \in \CC_5^-.$$

%

\subsection{Singular points, equation}
Of course, the equation giving $\XC/W'$ and the coordinates of the singular points depend
on a model for $W$ and on a choice of fundamental invariants. With the choices made in
Appendix~\ref{appendix:magma}, we get:

\bigskip

\begin{lem}\label{lem:singular points}
The coordinates of the singular points in $\PM(5,3,4,12)$ are given by
%
%
%
%
%

\medskip

\centerline{$p_5=[1:0:0:0],\quad p_9,p_9' =[0:1:0: \pm \frac{27}{64} \sqrt{-3}],\quad
p_{12},p_{12}'=[0:0:1:\pm \frac{16}{243}\sqrt{-3}]$}

\medskip

\centerline{$p_1,p_1' = [ \frac{1}{9}(\eta 536 \sqrt{19} + 2336) : \frac{1}{9}(-\eta 60\sqrt{19} - 260) :
\frac{1}{3}(\eta 130\sqrt{19} + 565) :0]$, with $\eta=\pm 1$,}

\medskip

\centerline{$p_2=[4/9 : 10/9 : -5/12 : 0]$.}
\end{lem}

\bigskip

\begin{proof}
The coordinates of $p_5$ are given by Lemma~\ref{lem:trivial}. Let $v_9$ (resp. $v_{12}$)
be a generator of the line $\Ker(w_9-\z_9 \Id_V)$ (resp. $\Ker(w_{12}-\z_{12}\Id_V)$).
The computation of the images $p_9$ and $p_{12}$ of $[v_9]$ and $[v_{12}]$ are done in
Computation~\ref{comp:p9p12} (note that we already know that $f_5(v_9)=f_5(v_{12})=f_9(v_{12})=f_{12}(v_9)=0$).
To renormalize $p_9$ (resp. $p_{12}$),
we just need to compute the corresponding value of $j^2/y_3^8$ (resp. $j^2/y_4^6$) and
we get the result.

For the coordinates of $p_1'$ and $p_1''$, we compute the image of $\XC^{v_2}$
and renormalize the solutions (see Computation~\ref{comp:p1}). We do a similar work for $p_2$, by using the
set $A$ involved in the proof of Lemma~\ref{lem:order 3} (see Computation~\ref{comp:p2}).
\end{proof}

\bigskip

Let us complete this result by computing the coordinates of the intersection point $p$
of $\CC_5^+$ and $\CC_5^-$. For this, note that $f_5(x)=\Jac(x)=0$ for all $x \in \XC^{v_3}$
(see the first two lines of Computation~\ref{comp:p}). Therefore, $p$ is the image of $\XC^{v_3}$ under
the map $\pi_\fb'$. This image is computed in Computation~\ref{comp:p} and we get
\equat\label{eq:p}
p:=[0 : 2/9 : 1/4 : 0].
\endequat

We are now ready to compute the equation of $\XC/W'$. A direct way would be to
compute the polynomial $P_\fb$ involved in~\eqref{eq:quotient w'}. However, this is a
too long computation and our computer collapsed... So we need another approach.

First, it follows from Lemma~\ref{prop:c5} and its proof that
there exist $\a$ and $\b$ in $\CM$ such that
$$Q_\fb(y_5,y_3,y_4) \equiv (\a y_3^4+\b y_4^3)^2 \mod y_5.$$
To obtain the coefficients $\a$ and $\b$, we just apply this equation to the coordinates of
$p_9$, $p_9'$, $p_{12}$, $p_{12}'$ and $p$ computed in Lemma~\ref{lem:singular points} and
in~\eqref{eq:p}. We then get
$$Q_\fb(y_5,y_3,y_4) \equiv -3\Bigl(\text{\small $\frac{27}{64}$} \,y_3^4-\text{\small $\frac{16}{243}$}\, y_4^3\Bigr)^2 \mod y_5.$$
By investigating all possible monomials of degree $19$ in $y_5$, $y_3$, $y_4$,
we get that there exist complex numbers $a$, $b$, $c$, $d$, $e$ such that
$$Q_\fb(y_5,y_3,y_4) \equiv -3\Bigl(\text{\small $\frac{27}{64}$} \,y_3^4-
\text{\small $\frac{16}{243}$}\, y_4^3\Bigr)^2 - y_5 R$$
with
$$R=y_3 y_4 (a y_3^4 + b y_4^3) +
c y_5 y_3^2 y_4^2 +d y_5^2 y_3^3 +e y_5^4 y_4.$$
Using the fact that $p_1$, $p_1'$, $p_2$ belong to $\XC/W'$, this gives three equations
for $a$, $b$, $c$, $d$, $e$. Computing the image $\XC^{s_2} \cap \PM(\Ker(x_1))$ in $\XC/W'$ gives
two Galois conjugate points $r$ and $r'$
with coordinates in $\QM(\sqrt{6})$ (see Computation~\ref{comp:r}). This leads to two
other equations for $a$, $b$, $c$, $d$, $e$. Solving this system with five equations and five
unknowns (see Computation~\ref{comp:equation})
yields a unique solution, which is given by
$(a,b,c,d,e)=(\frac{207}{32}, \frac{800}{729}, \frac{1375}{81}, -\frac{3125}{864}, -\frac{3125}{108})$.
In other words:
\equat\label{eq:equation}
\begin{array}{c}
\XC/W'\!=\!\{[y_5 : y_3 : y_4 : j] \in \PM(5,3,4,12)~|~j^2= - 3
(\frac{27}{64} \,y_3^4-\frac{16}{243}\, y_4^3)^2 - y_5 R(y_5,y_3,y_4)\}
\end{array}
\endequat
with
\equat\label{eq:equation R}
\begin{array}{c}
R=y_3 y_4 (\frac{207}{32}\, y_3^4 +
\frac{800}{729}\, y_4^3) +
\frac{1375}{81}\, y_5 y_3^2 y_4^2
- \frac{3125}{864}\, y_5^2 y_3^3
- \frac{3125}{108}\, y_5^3 y_4.
\end{array}
\endequat

\subsection{Affine charts}\label{sub:affine charts}
In the sequel, we will need to compute in the standard open affine charts of $\PM(5,3,4,12)$.
If $k \in \{3,4,5\}$, we denote by $\UC_k$ the open subset defined by
$$\UC_k=\{[y_5:y_3:y_4:j] \in \PM(5,3,4,12)~|~y_k \neq 0\}.$$
Since $\XC/W'$ is contained
in $\UC_5 \cup \UC_3 \cup \UC_4$, we do not need to work in the affine chart defined by $j \neq 0$.

\medskip

\subsubsection{The affine chart $\UC_3$}
Let us first start with $\UC_3$. Setting $y_3=1$, the new variables are $a_\trois=y_5^3$, $b_\trois=y_4^3$,
$c_\trois=y_5y_4$ and $j$. So
$$\UC_3=\{(a_\trois,b_\trois,c_\trois,j) \in \AM^4(\CM)~|~c_\trois^3=a_\trois b_\trois\}.$$
Then it follows from~\eqref{eq:equation} that
\equat\label{eq:u3}
\begin{array}{l}
(\XC/W') \cap \UC_3 =
\{(a_\trois,b_\trois,c_\trois,j) \in \AM^4(\CM)~|~c_\trois^3=a_\trois b_\trois
~\text{and}~ \petitespace\\
\hskip1.5cm \petitespace j^2= - 3(\frac{27}{64}-\frac{16}{243}\, b_\trois)^2 -
\frac{207}{32} c_\trois - \frac{800}{729} b_\trois c_\trois - \frac{1375}{81} c_\trois^2
+ \frac{3125}{864} a_\trois + \frac{3125}{108} a_\trois c_\trois \}.
\end{array}
\endequat

\subsubsection{The affine chart $\UC_4$}
Let us now turn to $\UC_4$. Setting $y_4=1$, the new variables are $a_\quatre=y_5^4$, $b_\quatre=y_3^4$,
$c_\quatre=y_5y_3$ and $j$. So
$$\UC_4=\{(a_\quatre,b_\quatre,c_\quatre,j) \in \AM^4(\CM)~|~c_\quatre^4=a_\quatre b_\quatre\}.$$
Then it follows from~\eqref{eq:equation} that
\equat\label{eq:u4}
\begin{array}{l}
(\XC/W') \cap \UC_4 =
\{(a_\quatre,b_\quatre,c_\quatre,j) \in \AM^4(\CM)~|~c_\quatre^4=a_\quatre b_\quatre
~\text{and}~ \petitespace\\
\hskip1.5cm \petitespace j^2= - 3(\frac{27}{64} b_\quatre -\frac{16}{243})^2 -
\frac{207}{32} b_\quatre c_\quatre - \frac{800}{729} c_\quatre - \frac{1375}{81} c_\quatre^2
+ \frac{3125}{864} c_\quatre^3 + \frac{3125}{108} a_\quatre \}.
\end{array}
\endequat

\subsubsection{The affine chart $\UC_5$}
Obtaining the equation of $(\XC/W') \cap \UC_5$ is much more complicated.
This is done in Appendix~\ref{appendix:magma} (more precisely, in~\S\ref{sub:u5}).
The result is given by Computation~\ref{comp:xwu5 dans a6}: setting $y_5=1$ and
$$Y_3=y_3^5,\quad Y_4=y_4^5,\quad h_1=y_3j,\quad h_2=y_3^2 y_4,\quad h_3=y_4^2 j\quad\text{and}\quad
h_4=y_3 y_4^3,$$
we have that $(\XC/W') \cap \UC_5$ may be identified with the closed
subvariety of $\AM^6(\CM)$ consisting of points $(Y_3,Y_4,h_1,h_2,h_3,h_4)$ such that
\equat\label{eq:xwu5}
\begin{cases}
h_1 h_4 = h_2 h_3,\\
Y_4 h_2 = h_4^2,\\
Y_4 h_1 = h_3 h_4,\\
Y_3 h_4 = h_2^3,\\
Y_3 h_3 = h_1 h_2^2,\\
Y_3 Y_4 = h_2^2 h_4,\\
Y_3^2 + \frac{2944}{243} Y_3 h_2 - \frac{2048}{6561} Y_3 h_4
- \frac{400000}{59049} Y_3 + \frac{4096}{2187} h_1^2 \\
\hskip1cm + \frac{5632000}{177147} h_2^2
+ \frac{3276800}{1594323} h_2 h_4
- \frac{3200000}{59049} h_2 +
    \frac{1048576}{43046721} h_4^2=0,\\
Y_3 h_2 h_4 + \frac{1048576}{43046721} Y_4^2 + \frac{3276800}{1594323} Y_4 h_4
- \frac{3200000}{59049} Y_4 \\ \hskip1cm + \frac{2944}{243} h_2^2 h_4 - \frac{2048}{6561} h_2 h_4^2
- \frac{400000}{59049} h_2 h_4 + \frac{4096}{2187} h_3^2 + \frac{5632000}{177147} h_4^2=0,\\
Y_3 h_2^2 + \frac{2944}{243} Y_3 h_4 + \frac{1048576}{43046721}Y_4 h_4 + \frac{4096}{2187} h_1 h_3 \\
\hskip1cm - \frac{2048}{656} h_2^2 h_4 - \frac{400000}{59049} h_2^2 + \frac{5632000}{177147} h_2 h_4
+ \frac{3276800}{1594323} h_4^2 - \frac{3200000}{59049} h_4=0.
\end{cases}
\endequat

\bigskip

\subsection{Two other smooth rational curves in ${\boldsymbol{\XC/W'}}$}
For $k \in \{3,4\}$, we set
$$\CC_k=\{[y_5:y_3:y_4:j] \in \XC/W'~|~y_k=0\}.$$

\medskip

\begin{prop}\label{prop:c3-c4}
The schemes $\CC_3$ and $\CC_4$ are reduced, irreducible and are smooth rational curves.
They intersect transversely at $p_5=[1 : 0 : 0 : 0]$.
\end{prop}

\bigskip

\begin{proof}
From the equation~\eqref{eq:equation}, we get that
$$
\begin{array}{c}
\CC_3 \simeq \{[y_5 : y_4 : j] \in \PM(5,4,12)~|~j^2= -\frac{2^8}{3^9}\, y_4^6 +
\frac{3125}{108} \, y_5^4 y_4\}.
\end{array}$$
But $\PM(5,4,12)=\PM(5,1,3)$, so
$$
\begin{array}{c}
\CC_3 \simeq \{[y_5 : y_4 : j] \in \PM(5,1,3)~|~j^2= -\frac{2^8}{3^9}\, y_4^6 +
\frac{3125}{108}\, y_5 y_4\}.
\end{array}$$
The open subset $\CC_3 \cap \UC_4$ is clearly isomorphic to $\AM^1(\CM)$.

Similary, from the equation~\eqref{eq:equation}, we get that
$$
\begin{array}{c}
\CC_4 \simeq \{[y_5 : y_3 : j] \in \PM(5,3,12)~|~j^2= - \frac{3^7}{2^{12}}\, y_3^6 +
\frac{3125}{864} \, y_5^3 y_3^3\}.
\end{array}$$
But $\PM(5,3,12)=\PM(5,1,4)$, so
$$
\begin{array}{c}
\CC_4 \simeq \{[y_5 : y_3 : j] \in \PM(5,1,4)~|~j^2= -\frac{3^7}{2^{12}}\, y_3^6 +
\frac{3125}{864}\, y_5 y_3^3\}.
\end{array}$$
The open subset $\CC_4 \cap \UC_3$ is clearly isomorphic to $\AM^1(\CM)$.

\medskip

So it remains to prove that $\CC_3$ and $\CC_4$ are smooth and intersect transversely.
Both questions are local around the point $p_5=[1:0:0:0]$, so we must work in the
affine chart $\UC_5$. Keep the notation from~\eqref{eq:xwu5}.
Then $\CC_3 \cap \UC_5$ (resp. $\CC_4 \cap \UC_5$)
is the subvariety of $(\XC/W') \cap \UC_5$ defined by $Y_3=h_1=h_2=h_4=0$
(resp. $Y_4=h_2=h_3=h_4=0$). Using Equations~\eqref{eq:xwu5}, we get
$$\CC_3 \cap \UC_5 = \{(Y_3,Y_4,h_1,h_2,h_3,h_4) \in \AM^6(\CM)~|~
Y_3=h_1=h_2=h_4= \l_\trois Y_4^2 + \mu_\trois Y_4 + \nu_\trois h_3^2 =0\}$$
and
$$\CC_4 \cap \UC_5 = \{(Y_3,Y_4,h_1,h_2,h_3,h_4) \in \AM^6(\CM)~|~
Y_4=h_2=h_3=h_4= \l_\quatre Y_3^2 + \mu_\quatre Y_3 + \nu_\quatre h_1^2 =0\}$$
for some non-zero rational numbers $\l_\trois$, $\mu_\trois$, $\nu_\trois$, $\l_\quatre$, $\mu_\quatre$
and $\nu_\quatre$. The smoothness of $\CC_3 \cap \UC_5$ and $\CC_4 \cap \UC_5$ follows
immediately, as well as the fact they intersect transversely at $p_5$ (which
corresponds to the point $(0,0,0,0,0,0) \in \AM^6(\CM)$).
\end{proof}

\bigskip

\section{The K3 surface ${\boldsymbol{\XCt}}$}\label{sec:xtilde}

\medskip

Recall that $\r : \XCt \longto \XC/W'$ denotes the minimal smooth resolution.
We will deduce several properties of $\XCt$ (Picard lattice, elliptic fibration,...)
from the list of properties of $\XC/W'$ given in the previous section.
Note that since $\XCt$ is obtained from $\XC/W'$ by successively blowing-up the singular
locus, the automorphism $\s$ of $\XC/W'$ lift to an automorphism of $\XCt$ (which will still
be denoted by $\s$).

We denote by $\D_1$ and $\D_1'$ the two smooth rational
curves of $\XCt$ lying above $p_1$ and $p_1'$ respectively.
For $e \in \{2,5,9,12\}$, we denote by $\D_e^1$,\dots, $\D_e^{r_e}$ the smooth rational curves
of $\XCt$ lying above $p_e$ (here, $r_e$ is the Milnor number of the singularity
$p_e$\footnote{For the definition of the Milnor number of an isolated hypersurface singularity,
see~\cite[\S{7}]{milnor}: recall that the Milnor number of a singularity of type $A_k$, $D_k$ or $E_k$ is equal to $k$.}),
and we assume that they are numbered in such a way that $\D_e^k \cap \D_e^{k+1} \neq \vide$.
For $e \in \{9,12\}$,
the smooth rational curves of $\XCt$ lying above $p_e'$ are then given
by $\s(\D_e^1)$,\dots, $\s(\D_e^{r_e})$.

Finally, we denote by $\CCt_5^\pm$ the strict transform of $\CC_5^\pm$
in $\XCt$. Of course, $\CCt_5^-=\s(\CCt_5^+)$.
As $\XCt$ is obtained from $\XC/W'$ by successive blow-ups of points,
$\CCt_5^+$ and $\CCt_5^-$ are smooth rational curves.
Also, we denote by $\CCt_3$ and $\CCt_4$ the strict transforms of $\CC_3$
and $\CC_4$: for the same reason, they are also smooth rational curves.

One of the aims of this section is to determine the intersection graph of the 22 smooth rational
curves $\CCt_3$, $\CCt_4$, $\CCt_5^\pm$, $\D_1$, $\D_1'$,
$(\D_e^k)_{e \in \{2,5,9,12\}, 1 \le k \le r_e}$ and
$(\lexp{\s}{\D_e^k})_{e \in \{9,12\}, 1 \le k \le r_e}$. For this, we will use
the construction of an elliptic fibration on $\XCt$.

%
%

%
Recall that, for a K3 surface, an {\it elliptic fibration} is just a
morphism to $\PM^1(\CM)$ such that at least one fiber is a smooth elliptic curve.
Since $\XC/W'$ has $A_4 + 2\, A_3 + 3\, A_2 + 2\, A_1$ singularities, its Picard number
is greater than or equal to $1+(4 + 2 \cdot 3 + 3 \cdot 2 + 1)=19$ (in fact, we will see
later that it has Picard number $20$). Therefore, it admits an elliptic fibration
(because every K3 surface with Picard number $\ge 5$ admits an elliptic
fibration~\cite[Chap.~11,~Prop.~1.3(ii)]{huybrechts}).
Another aim of this section is to contruct at least one such fibration. For this,
let
$$\fonction{\ph}{(\XC/W') \setminus\{p_5\}}{\PM^1(\CM)}{[y_5 : y_3 : y_4 : j]}{[y_3^4 : y_4^3].}$$
Then $\ph$ is a well-defined morphism of varieties,
so the map $\ph \circ \r : \XCt \setminus \r^{-1}(p_5) \longto \PM^1(\CM)$ is also
a well-defined morphism of varieties. We are now ready to prove the second main result
of our paper:

\bigskip

\begin{theo}\label{theo:elliptic}
With the above notation, we have:
\begin{itemize}
\itemth{a} The morphism $\ph \circ \r : \XCt \setminus \r^{-1}(p_5) \longto \PM^1(\CM)$ extends to a unique
morphism
$$\pht : \XCt \to \PM^1(\CM),$$
which is an elliptic fibration whose singular fibers are of type $E_7+E_6+A_2+2\,A_1$.

\itemth{b} There exists a way of numbering the smooth rational
curves lying above singular points of $\XC/W'$
such that the intersection graph of the 22 smooth rational curves
$\CCt_3$, $\CCt_4$, $\CCt_5^\pm$, $\D_1$, $\D_1'$,
$(\D_e^k)_{e \in \{2,5,9,12\}, 1 \le k \le r_e}$ and
$(\lexp{\s}{\D_e^k})_{e \in \{9,12\}, 1 \le k \le r_e}$ is given
by
$$
\begin{picture}(360,170)
\put( 55,140){\circle{10}}\put(33,136){$\D_9^1$}
\put(105,140){\circle*{10}}\put(110,150){$\CCt_5^+$}
\put(155,140){\circle{10}}\put(147,150){$\D_{12}^1$}
\put(205,125){\circle{10}}\put(197,135){$\D_{12}^2$}
\put(255,110){\circle{10}}\put(247,120){$\D_{12}^3$}
\put( 55,110){\circle{10}}\put(33,106){$\D_9^2$}
\put(105,110){\circle{10}}\put(112,106){$\D_5^4$}
\put( 55, 80){\circle{10}}\put(33,76){$\CCt_4$}
\put(105, 80){\circle{10}}\put(97,64){$\D_5^3$}
\put(155, 80){\circle{10}}\put(155,80){\circle*{6}}\put(147,64){$\D_5^2$}
\put(205, 80){\circle{10}}\put(197,64){$\D_5^1$}
\put(255, 80){\circle{10}}\put(265,76){$\CCt_3$}
\put( 55, 50){\circle{10}}\put(28,46){$\lexp{\s}{\D_{9}^2}$}
\put( 55, 20){\circle{10}}\put(28,16){$\lexp{\s}{\D_{9}^1}$}
\put(105, 20){\circle*{10}}\put(110,3.3){$\CCt_5^-$}
\put(155, 20){\circle{10}}\put(147,4){$\lexp{\s}{\D_{12}^1}$}
\put(205, 35){\circle{10}}\put(197,19){$\lexp{\s}{\D_{12}^2}$}
\put(255, 50){\circle{10}}\put(247,34){$\lexp{\s}{\D_{12}^3}$}

\put(305,110){\circle{10}}\put(298.5,120){$\D_{2}^1$}
\put(355,110){\circle{10}}\put(348.5,120){$\D_{2}^2$}
\put(330,80){\circle{10}}\put(340,76){$\D_1$}
\put(330,50){\circle{10}}\put(340,46){$\D_1'$}
\put( 60, 140){\line(1,0){40}}
\put( 55, 135){\line(0,-1){20}}
\put(110, 140){\line(1,0){40}}
\put( 55, 105){\line(0,-1){20}}
\put( 55, 75){\line(0,-1){20}}
\put( 55, 45){\line(0,-1){20}}
\put(105, 105){\line(0,-1){20}}
\put(255, 105){\line(0,-1){20}}
\put(255, 75){\line(0,-1){20}}
\put( 60, 80){\line(1,0){40}}
\put(110, 80){\line(1,0){40}}
\put(160, 80){\line(1,0){40}}
\put(210, 80){\line(1,0){40}}
\put( 60, 20){\line(1,0){40}}
\put(110, 20){\line(1,0){40}}

\put(310,110){\line(1,0){40}}

\qbezier(159.789,138.563)(180,132.5)(200.211,126.437)
\qbezier(209.789,123.563)(230,117.5)(250.211,111.437)
\qbezier(159.789,21.437)(180,27.5)(200.211,33.563)
\qbezier(209.789,36.563)(230,42.5)(250.211,48.437)

\put(60,140){\oval(90,50)[t]}
\put(60,20){\oval(90,50)[b]}
\put(15,140){\line(0,-1){120}}

\put(-36,78){${\boldsymbol{(\bigstar)}}$}

\end{picture}
$$
In this graph:
\begin{itemize}
\itemth{b1} The union of the singular fibers of $\pht$ of type $E_7$ and $E_6$ is given
by the white disks in the big connected subgraph of $(\bigstar)$.

\itemth{b2} The singular fibers of $\pht$ of type $A_1$ are
$\CCt_1 \cup \D_1$ and $\CCt_1' \cup \D_1'$ for some smooth rational curves $\CCt_1$ and $\CCt_1'$.

\itemth{b3} The singular fiber of $\pht$ of type $A_2$ is $\CCt_2 \cup \D_2^1 \cup \D_2^2$ for some
smooth rational curve $\CCt_2$.

\itemth{b4} The curves marked with full black disks in $(\bigstar)$ are sections of $\pht$.

\itemth{b5} The curve $\D_5^2$ is a double section of $\pht$.
\end{itemize}

\itemth{c} The $22$ smooth rational curves in this intersection graph generate
the Picard lattice $\Pic(\XCt)$. More precisely, $\Pic(\XCt)$ is generated by the list
obtained from these $22$ smooth rational curves by removing $\D_5^2$ and $\D_5^4$.
Its discriminant is $-228=-2^2 \cdot 3 \cdot 19$.

\itemth{d} The Mordell-Weil group of $\pht$ is isomorphic
to $\ZM^2$.

\itemth{e} The transcendental lattice is given by the matrix
$\begin{pmatrix}
2 & 0 \\
0 & 114
\end{pmatrix}$.
\end{itemize}
\end{theo}

\bigskip

\begin{rema}
There are several possible types of singular fibers of type $A_1$ (resp. $A_2$) in elliptic
fibrations. In the above Theorem~\ref{theo:elliptic}, singular fibers of type $A_1$ (resp. $A_2$)
are of type $\Irm_2$ (resp. $\Irm_3$) in Kodaira's classification.\finl
\end{rema}

\bigskip

\begin{proof}
The (very computational) proof of the statements~(a) and~(b) is given
in Appendix~\ref{appendix:open}.

For~(c), let $M$ denote the incidence
matrix of the $18$ smooth rational curves belonging to the big connected
subgraph of~$(\bigstar)$. Then $M$ has rank $16$ and the greatest common divisor
of the diagonal minors of $M$ is equal to $19$ by Computation~\ref{comp:incidence}.
Moreover, the diagonal minor corresponding to the curves $\CCt_3$, $\CCt_4$, $\CCt_5^\pm$, $\D_5^1$,
$\D_5^3$, $(\D_e^k)_{e \in \{9,12\}, 1 \le k \le r_e}$ and
$(\lexp{\s}{\D_e^k})_{e \in \{9,12\}, 1 \le k \le r_e}$ is equal to $-19$
(see the last command of Computation~\ref{comp:incidence}).
So, if we denote by $\L$ the lattice generated by these $16$ curves together with
$\D_1$, $\D_1'$, $\D_2^1$ and $\D_2^2$,
then $\L$ has rank $20$ and discriminant $-228=-2^2 \cdot 3 \cdot 19$.
This shows that $\Pic(\XCt)$ has rank $\ge 20$, and so has rank $20$ as
a K3 surface has always Picard number $\le 20$. If we denote by $n$ the index of $\L$
in $\Pic(\XCt)$, then $n^2$ divides $228$, which shows that $n \in \{1,2\}$.

But if we denote by $\Trm(\XCt)$ the transcendental lattice of $\XCt$, then $\Trm(\XCt)$
has rank $22-20=2$, is even and definite positive, with discriminant $228/n^2$. Hence it can be represented
by a matrix of the form
$$\begin{pmatrix}
2a  & b  \\
b  & 2c
\end{pmatrix}$$
with $a$, $c > 0$ and $4ac - b^2=228/n^2$. But $4ac-b^2 \equiv 0$ or $3 \mod 4$,
so $4ac-b^2 \neq 57$. This shows that $n=1$ and that
$\Pic(\XCt)$ is generated by $\CCt_3$, $\CCt_4$, $\CCt_5^\pm$, $\D_5^1$,
$\D_5^3$, $(\D_e^k)_{e \in \{9,12\}, 1 \le k \le r_e}$ and
$(\lexp{\s}{\D_e^k})_{e \in \{9,12\}, 1 \le k \le r_e}$,
as expected. This concludes the proof of~(c).

\medskip

(d) Since we have determined the Picard lattice of $\XCt$ in~(c), the
structure of the Mordell-Weil group follows (note that it has no torsion,
which is compatible with~\cite[Table 1, entry~2420]{shimada}).

\medskip

(e) The transcendental lattice of $\XCt$ is given by a matrix of the form
$\begin{pmatrix} 2a & b \\ b & 2c \end{pmatrix}$ whose underlying quadratic
form is definite positive and has discriminant $228$ by~(c). The classification
of even integral binary quadratic forms~\cite[Theo.~2.3]{buell},  shows that there are only four
such matrices, up to equivalence, namely:
$$M_1=\begin{pmatrix} 2 & 0 \\ 0 & 114 \end{pmatrix},
\quad M_2=\begin{pmatrix} 6 & 0 \\ 0 & 38 \end{pmatrix},\quad
M_3=\begin{pmatrix} 4 & 2 \\ 2 & 58 \end{pmatrix}\quad\text{and}\quad
M_4=\begin{pmatrix} 12 & 6 \\ 6 & 22 \end{pmatrix}.$$
Let $P=\Pic(\XCt)$ and $T=\Trm(\XCt)$. Let
$$P^\perp = \{v \in \QM \otimes_\ZM P~|~\forall\, v' \in P,~\langle v,v' \rangle \in \ZM\}$$
and let us define $T^\perp$ similarly. Then the quadratic forms on $\QM \otimes_\ZM P$
and $\QM \otimes_\ZM T$
induce well-defined maps
$$q_P : P^\perp/P \longto \QM/2\ZM\qquad \text{and}\qquad q_T : T^\perp/T \longto \QM/2\ZM.$$
Since $\Hrm^2(\XCt,\ZM)$ is unimodular of signature $(3,19)$, it turns out that there is an isomorphism
$\iota : T^\perp/T \longiso P^\perp/P$ such that $q_T=-q_P \circ \iota$
(see~\cite[Prop.~1.6.1]{nikulin}).
In particular, the set of values of $q_T$ and $-q_P$ coincide: the subset $-q_P(P^\perp/P)$ of $\QM/2\ZM$
can easily be computed thanks to~(b) and~(c) using {\sc Magma}, and we only need to compare
the corresponding sets for the four rank $2$ lattices determined by $M_1$, $M_2$, $M_3$
and $M_4$. This comparison gives the result (see~\S\ref{sub:theo e} for details about these computations).
\end{proof}

\bigskip

\begin{rema}\label{rem:kummer}
If we denote by
$\begin{pmatrix}
2a  & b  \\
b  & 2c
\end{pmatrix}$ a matrix representating the transcendental lattice of $\XCt$,
then $4ac-b^2=228$, so $b$ is even, Writing $b=2\b$, we get
$ac = 57 + \b^2$, so $ac \equiv 1$ or $2 \mod 4$. In particular,
$a$ or $c$ is odd. This shows that the K3 surface $\XCt$ is not a Kummer surface~\cite[Chap.~14,~Cor.~3.20]{huybrechts}.

Also, $\XCt$ is not isomorphic to any of the singular K3 surfaces
constructed by Barth-Sarti in~\cite{barth sarti} (their transcendental lattices have been computed
by Sarti in~\cite{sarti} and none of them have discriminant $228$). For the same reason, $\XCt$ is not isomorphic
to any of the singular K3 surfaces constructed by Brandhorst-Hashimoto in~\cite{brandhorst}
(see also~\cite{bonnafe sarti m20} for a description of some of these).\finl
\end{rema}

\bigskip

\section{Complement: action of ${\boldsymbol{W}}$ on the cohomology of ${\boldsymbol{\XC}}$}\label{sec:action}

\medskip

The group $W$ acts on $\XC$ so it acts on the cohomology groups $\Hrm^k(\XC,\CM)$.
Since $\XC$ is a complete intersection in $\PM^5(\CM)$, with defining equations
of degree $2$, $6$, and $8$, we have:
\equat\label{eq:coho}
\dim_\CM \Hrm^k(\XC,\CM) =
\begin{cases}
1 & \text{if $k \in \{0,4\}$,}\\
9\,502 & \text{if $k=2$,}\\
0 & \text{otherwise.}
\end{cases}
\endequat
The action of $W$ on $\Hrm^0(\XC,\CM)$ and $\Hrm^4(\XC,\CM)$ is trivial.
The aim of this subsection is to determine the character of the representation
of $W$ afforded by $\Hrm^2(\XC,\CM)$.

For this, we first need to parametrize the irreducible characters of $W$.
If $\chi \in \Irr(W)$, we denote by $b_\chi$ the minimal number $k$ such that
$\chi$ occurs in the character of the symmetric power $\Sym^k(V)$ of the natural
representation $V$ of $W$. For instance, if we denote by $\unb_W$ the trivial
character of $W$ and by $\chi_V$ the character afforded by the natural
representation $V$, then
\equat\label{eq:b}
b_{\unb_W}=0,\qquad b_{\chi_V}=1 \qquad\text{and}\qquad b_\e=|\AC|=36
\endequat
(recall that $\e$ denotes the restriction of the determinant to $W$).
Indeed, the first two equalities are immediate from the definition and the last one follows
from~\cite[Chap.~V,~\S{5},~Prop.~5]{bourbaki} and~\eqref{eq:ref}.
Recall from Molien's formula that the number $b_\chi$ an be computed as follows:
let $t$ be an indeterminate and let
$$F_\chi(t)=\frac{\prod_{d \in \Deg(W)} (1-t^d)}{|W|}
\sum_{w \in W} \frac{\chi(w^{-1})}{\det(1-tw)} \in \CM(t).$$
It is a classical fact~\cite[\S{4.5.2}]{broue} that $F_\chi(t) \in \NM[t]$, that $F_\chi(1)=\chi(1)$ and
\equat\label{eq:bb}
b_\chi=\val F_\chi(t).
\endequat
The polynomial $F_\chi(t)$ is called the {\it fake degree} of $\chi$.

\def\DBrm{{\mathrm{DB}}}
\def\DBCC{{\mathcal{DB}}}

A particular feature of the Weyl group of type $\Erm_6$ is that the map
\equat\label{eq:irr}
\fonction{\DBrm}{\Irr(W)}{\NM \times \NM}{\chi}{(\chi(1),b_\chi)}
\endequat
is injective (see for instance the first command of Computation~\ref{comp:ew}).
We denote by $\DBCC(W)$ the image of $\DBrm$ and, if $(d,b) \in \DBCC(W)$,
we denote by $\phi_{d,b}$ its inverse image in $\Irr(W)$. Note that $\phi_{d,b}$ is the
character afforded by an irreducible representation of dimension $d$.
For instance, by~\eqref{eq:b},
we get
\equat\label{eq:phidb}
\phi_{1,0}=\unb_W,\qquad \phi_{6,1}=\chi_V\qquad\text{and}\qquad \phi_{1,36}=\e.
\endequat
By the last line of Computation~\ref{comp:ew}, we have that $|\Irr(W)|=25$ and that
\equat\label{eq:ew}
\begin{array}{c}
\DBCC(W)=\{( 1 , 0) ; ( 1 , 36); ( 6 , 1); ( 6 , 25); ( 10 , 9); ( 15 , 17); ( 15 , 4); ( 15 , 16); ( 15 , 5);\\
( 20 , 20); ( 20 , 10); ( 20 , 2); ( 24 , 6); ( 24 , 12); ( 30 , 3); ( 30 , 15); ( 60 , 11); \\
( 60 , 8); ( 60 , 5); ( 64 , 13); ( 64 , 4); ( 80 , 7); ( 81 , 6); ( 81 , 10); ( 90 , 8)\}.
\end{array}
\endequat

For $i \ge 0$, let $\chi_\XC^{(i)}$ denote the character afforded by the $W$-module
$\Hrm^i(\XC,\CM)$. We set
$$\chi_\XC=\sum_{i \ge 0} (-1)^i \chi_\XC^{(i)}.$$
By~\eqref{eq:coho}, we have
\equat\label{eq:chix}
\chi_\XC=\chi_\XC^{(0)} + \chi_\XC^{(2)} + \chi_\XC^{(4)} = 2 \cdot \unb_W + \chi_\XC^{(2)}.
\endequat
We determine $\chi_\XC^{(2)}$ in the next proposition:

\bigskip

\begin{prop}\label{prop:character}
With the above notation, we have
\eqna
\chi_\XC^{(2)}&=&
\unb_W + 3\,\e + 8\,\phi_{6, 25} + 2\,\phi_{ 10, 9} +
7 \,\phi_{ 15, 17} + \phi_{ 15, 4} + 9 \,\phi_{ 15, 16}
+ \phi_{ 15, 5} \\
&& + 14\,\phi_{ 20, 20} +
4 \,\phi_{ 20, 10} +
2 \,\phi_{ 24, 6}
+ 8 \,\phi_{ 24, 12} + 14\,\phi_{ 30, 15} +
18 \,\phi_{ 60, 11} + 12 \,\phi_{ 60, 8} \\
&& +
4 \,\phi_{ 60, 5} + 26 \,\phi_{ 64, 13} +
2 \,\phi_{ 64, 4} + 12 \,\phi_{ 80, 7} +
7\,\phi_{ 81, 6} + 21 \,\phi_{ 81, 10} +
12 \,\phi_{ 90, 8}.
\endeqna
\end{prop}

\bigskip

\begin{proof}
Since $\XC$ is smooth and $W$ is finite, it follows from Lefschetz fixed point formula that
$\chi_\XC(w)$ is equal to the Euler characteristic of the fixed point subvariety $\XC^w$.
If $\dim(\XC^w) \ge 1$, then $w$ is conjugate to $1$, $s_1$ or $w_3$ (see Computation~\ref{comp:fixed dim}).
But:
\begin{itemize}
\item[$\bullet$] $\chi_\XC(1)=9\,504$ by~\eqref{eq:coho}.

\item[$\bullet$] Note that $\PM(V)^{s_1}= [e_1] \cup
\PM(V^{s_1})$. Since $[e_1] \not\in \XC$ by Remark~\ref{rem:f2}, we have that
$$\XC^{s_1} = \XC \cap \PM(V^{s_1}).$$
So $\XC^{s_1}$ is a smooth complete intersection in $\PM(V^{s_1}) \simeq \PM^4(\CM)$
defined by equations of degree $2$, $6$ and $8$ (the restrictions of $f_2$, $f_6$ and
$f_8$ to $V^{s_1}$), so it has Euler characteristic
$-2 \cdot 6 \cdot 8 \cdot (2+6+8-4-1)=-1\,056$. Hence, $\chi_\XC(s_1)=-1\,056$.

\item[$\bullet$] Recall that $\XC^{w_3}$ is the disjoint union of two smooth
curves of genus $10$ by Lemma~\ref{lem:courbe 10}. So $\chi_\XC(w_3)=-36$.
\end{itemize}
If $\dim(\XC^w) \le 0$, then $\chi_\XC(w)$ is just the cardinality of $\XC^w$
(it might be equal to $0$).
These last values of $\chi_\XC$ as well as the decomposition
of $\chi_\XC$ as a sum of irreducible characters
are computed in Computation~\ref{comp:chix}. The result then follows from~\eqref{eq:chix}.
\end{proof}

\bigskip

To be fair, knowing the exact character is not that interesting, but at least
we will use it for making a sanity check for Proposition~\ref{prop:singularities X}.
Indeed, $\Hrm^k(\XC,\CM)^{W'}$ is the direct sum of $\Hrm^k(\XC,\CM)^W$ and
the $\e$-isotypic component of $\Hrm^k(\XC,\CM)$. Then Proposition~\ref{prop:character}
together with~\eqref{eq:coho} shows that
$$\sum_{k \in \ZM} (-1)^k \dim_\CM \Hrm^k(\XC,\CM)^{W'}=6.$$
In other words, the Euler characteristic of $\XC/W'$ is equal to $6$.
But the fiber of the map $\XCt \to \XC/W'$ above an $A_k$ singularity
is the union of $k$ smooth rational curves in $A_k$-configuration,
and this union has Euler characteristic $k+1$. So the Euler characteristic of $\XCt$
is the Euler characteristic of $\XC/W'$ plus the sum of all the Milnor numbers
of singularities of $\XC/W'$. So, by Proposition~\ref{prop:singularities X},
the Euler characteristic of $\XCt$ is
$$6 + 2 \cdot 1 + 3 \cdot 2 + 2 \cdot 3 + 4=24,$$
as expected for a K3 surface.

\bigskip

\begin{rema}\label{rem:hodge}
Since $\XC$ is a smooth complete intersection, its Hodge numbers
can be computed from the degrees of the equations and we get that
$$h^{2,0}(\XC)=h^{0,2}(\XC)=1\,591\qquad \text{and}\qquad h^{1,1}(\XC)=6\,320.$$
However, we do not know how to compute the character of the representations
$\Hrm^{p,q}(\XC,\CM)$ of $W$, for $p+q=2$.\finl
\end{rema}

\bigskip

\setcounter{section}{0}
\def\sectionname{Appendix}
\renewcommand\thesection{\Roman{section}}


\section{Magma computations for Sections~\ref{sec:k3} and~\ref{sec:xw}}\label{appendix:magma}

\medskip

\subsection{Set-up}
Before performing some computations, let us first define in {\sc Magma} the objects involved in this paper.

\bigskip

\begin{quotation}
{\small\begin{verbatim}
W:=PrimitiveComplexReflectionGroup(35);  // Weyl group of type E_6
V:=VectorSpace(DW);  // underlying representation : C^6
n:=Rank(W);  // n=6

DW:=DerivedSubgroup(W); // derived subgroup of W (i.e. W')
Q:=Rationals();

conjDW:=ConjugacyClasses(DW);
conjDW:=[i[3] : i in conj]; // rep. of conj. classes of W'

R:=InvariantRing(W);  // invariant ring C[V]^W
P<x1,x2,x3,x4,x5,x6>:=PolynomialRing(R);  // C[V]
P5:=Proj(P);  // P(V)

f2:=InvariantsOfDegree(W,2)[1];
f5:=InvariantsOfDegree(W,5)[1];
f6:=InvariantsOfDegree(W,6)[1];
f8:=InvariantsOfDegree(W,8)[1];
f9:=InvariantsOfDegree(W,9)[1];
f12:=InvariantsOfDegree(W,12)[1];

f:=[f2,f5,f6,f8,f9,f12];  // fundamental invariants
\end{verbatim}}
\end{quotation}

\bigskip

The first and the last lines require some explanation. For the first line,
we do not use the command {\tt ShephardTodd(35)} but another model of the Weyl group
of type $\Erm_6$ (see~\cite[Rem.~2.3]{bonnafe singular} for explanations). Also, the fact that {\tt f}
defines a set of fundamental invariants can be checked by showing that the
variety defined as their common zeroes is empty:

\bigskip

\begin{quotation}
{\small\begin{verbatim}
> Dimension(Scheme(P5,f));
-1
\end{verbatim}}
\end{quotation}

\bigskip

Finally, note that the variety $\XC$ is defined in {\sc Magma} as a scheme
over the field of coefficients of $W$, which is $\QM$. We will need sometimes
to look at points of $\XC$ with coordinates in bigger number fields.
Let us now compute the dimension of $\XC$ and show that its affine open
subset defined by $x_6 \neq 0$ is smooth:

\bigskip

\begin{comp}\label{comp:1}
~
\begin{quotation}
\begin{verbatim}
> X:=Scheme(P5,[f2,f6,f8]);   // the variety X=Z(f_2,f_6,f_8)
> Dimension(X);
2
> IsSingular(AffinePatch(X,1));
false
\end{verbatim}
\end{quotation}
\end{comp}

\bigskip

Finally, we will need an auxiliary function {\tt FixedPoints} which computes
the fixed points of an element of $\GL_\CM(V)$ for its action on $\PM(V)$:

\bigskip

\begin{quotation}
{\small\begin{verbatim}
FixedPoints:=function(g) local i,j,ii,k,res,mat;
  j:=[P.i^g : i in [1..n]];
  res:=[];
  for i in Subsets({1..n},2) do
    ii:=[k : k in i];
    mat:=Matrix(P,2,2,[[P.ii[1],P.ii[2]],[j[ii[1]],j[ii[2]]]]);
    res:=res cat [Determinant(mat)];
  end for;
  res:=MinimalBasis(Scheme(P5,res));
  res:=Scheme(P5,res);
  return res;
end function;
\end{verbatim}}
\end{quotation}

\bigskip

\subsection{Singularities of ${\boldsymbol{\XC/W'}}$}
Below are the necessary computations for completing the proof
of Proposition~\ref{prop:singularities X} and Lemma~\ref{lem:singular points}.

\begin{comp}\label{comp:order 2}~
\begin{quotation}
{\small\begin{verbatim}
> liste2:=[w : w in conjDW | Order(w) eq 2];
> # liste2;
2
> [Eigenvalues(w) : w in liste2];
[
    { <-1, 4>, <1, 2> },
    { <-1, 2>, <1, 4> }
]
> w2:=[w : w in liste2 | Dimension(Eigenspace(w,1)) eq 2];
> w2:=w2[1];
> v2:=[w : w in liste2 | Dimension(Eigenspace(w,1)) eq 4];
> v2:=v2[1];
\end{verbatim}}
\end{quotation}
\end{comp}

\begin{comp}\label{comp:w2}~
\begin{quotation}
{\small\begin{verbatim}
> Xw2:=FixedPoints(w2) meet X;
> Dimension(Xw2);
0
> Xw2 eq Xw2 meet Scheme(P5,[f5,f9]);
true
\end{verbatim}}
\end{quotation}
\end{comp}

\begin{comp}\label{comp:v2}~
\begin{quotation}
{\small\begin{verbatim}
> Xv2:=FixedPoints(v2) meet X;
> Xv2:=Scheme(P5,MinimalBasis(Xv2));
> Dimension(Xv2);  // 0
0
> Degree(Xv2);  // 96
96
> [g : g in DW | Dimension(Xv2 meet FixedPoints(g)) eq 0] eq [v2^0,v2];
true
> Cv2:=Centralizer(DW,v2);
> Order(Cv2);
96
\end{verbatim}}
\end{quotation}
\end{comp}

\begin{comp}\label{comp:order 3}~
\begin{quotation}
{\small\begin{verbatim}
> liste3:=[w : w in conjDW | Order(w) eq 3];
> # liste3;
4
> [Eigenvalues(ChangeRing(w,CyclotomicField(3))) : w in liste3];
[
    {
        <zeta_3, 3>,
        <-zeta_3 - 1, 3>
    },
    {
        <zeta_3, 3>,
        <-zeta_3 - 1, 3>
    },
    {
        <1, 4>,
        <zeta_3, 1>,
        <-zeta_3 - 1, 1>
    },
    {
        <zeta_3, 2>,
        <1, 2>,
        <-zeta_3 - 1, 2>
    }
]
> w3:=[w : w in liste3 | Dimension(Eigenspace(w,1)) eq 0];
> w3:=w3[1];
> IsConjugate(DW,w3,w3^-1);
false
> v3:=[w : w in liste3 | Dimension(Eigenspace(w,1)) eq 2];
> v3:=v3[1];
> u3:=[w : w in liste3 | Dimension(Eigenspace(w,1)) eq 4];
> u3:=u3[1];
> Dimension(FixedPoints(w3) meet X);
1
\end{verbatim}}
\end{quotation}
\end{comp}

\begin{comp}\label{comp:v3}~
\begin{quotation}
{\small\begin{verbatim}
> Xv3:=FixedPoints(v3) meet X;
> Xv3:=Scheme(P5,MinimalBasis(Xv3));
> Dimension(Xv3);
0
> testv3:=[w : w in Transversal(DW,sub<DW | v3>) |
  Dimension(Xv3 meet FixedPoints(w)) eq 0];
> # testv3;
3
> Gv3:=sub<DW | testv3 cat [v3]>;
> Order(Gv3);
9
> IsElementaryAbelian(Gv3);
true
> Set([Xv3 eq Xv3 meet FixedPoints(w) : w in Gv3]);
{ true }
\end{verbatim}}
\end{quotation}
\end{comp}

\begin{comp}\label{comp:u3}~
\begin{quotation}
{\small\begin{verbatim}
> Xu3:=FixedPoints(u3) meet X;
> Xu3:=Scheme(P5,MinimalBasis(Xu3));
> Dimension(Xu3);
0
> Degree(Xu3);
96
> Nu3:=Normalizer(DW,sub<DW | u3>);
> Order(Nu3);
216
> IXu3:=IrreducibleComponents(Xu3);
> IXu3:=[Scheme(P5,MinimalBasis(I)) : I in IXu3];
> [Degree(I) : I in IXu3];
[ 72, 12, 12 ]
\end{verbatim}}
\end{quotation}
\end{comp}

\begin{comp}\label{comp:A}~
\begin{quotation}
{\small\begin{verbatim}
> A:=[I : I in IXu3 | Degree(I) eq 72];
> A:=A[1];
> testA:=[w : w in Transversal(DW,sub<DW | u3>) |
  Dimension(A meet FixedPoints(w)) eq 0];
> # testA;
1
> Dimension(Scheme(A,[f5,f12]));
-1
\end{verbatim}}
\end{quotation}
\end{comp}

\begin{comp}\label{comp:B}~
\begin{quotation}
{\small\begin{verbatim}
> B:=[I : I in IXu3 | Degree(I) eq 12];
> B:=B[1] join B[2];
> testB:=[w : w in Transversal(DW,sub<DW | u3>) |
  Dimension(B meet FixedPoints(w)) eq 0];
> # testB;
5
> GB:=sub<DW | testB cat [u3]>;
> Order(GB);
27
> Set([Dimension(FixedPoints(w) meet B) : w in GB]);
{ -1, 0 }
> IsElementaryAbelian(GB);
true
\end{verbatim}}
\end{quotation}
\end{comp}

\begin{comp}\label{comp:order 4}~
\begin{quotation}
{\small\begin{verbatim}
> liste4:=[w : w in conjDW | Order(w) eq 4];
> # liste4;
2
> [Eigenvalues(ChangeRing(w,CyclotomicField(4))) : w in liste4];
[
    {
        <1, 2>,
        <zeta_4, 2>,
        <-zeta_4, 2>
    },
    {
        <-zeta_4, 1>,
        <1, 2>,
        <-1, 2>,
        <zeta_4, 1>
    }
]
> w4:=[w : w in liste4 | Dimension(Eigenspace(w,-1)) eq 0];
> w4:=w4[1];
> v4:=[w : w in liste4 | Dimension(Eigenspace(w,-1)) eq 2];
> v4:=v4[1];
\end{verbatim}}
\end{quotation}
\end{comp}

\begin{comp}\label{comp:w4 v4}~
\begin{quotation}
{\small\begin{verbatim}
> Xw4:=FixedPoints(w4) meet X;
> Xw4 eq Xw4 meet Scheme(P5,[f5,f9]);
true
> Xv4:=FixedPoints(v4) meet X;
> Dimension(Xv4);
-1
\end{verbatim}}
\end{quotation}
\end{comp}

\begin{comp}\label{comp:order 5}~
\begin{quotation}
{\small\begin{verbatim}
> liste5:=[w : w in conjDW | Order(w) eq 5];
> # liste5;
1
> w5:=liste5[1];
> Order(Centralizer(DW,w5));
5
> Xw5:=FixedPoints(w5) meet X;
> Dimension(Xw5);
0
> Degree(Xw5);
4
\end{verbatim}}
\end{quotation}
\end{comp}

\begin{comp}\label{comp:order 6}~
\begin{quotation}
{\small\begin{verbatim}
> liste6:=[w : w in conjDW | Order(w) eq 6];
> # liste6;
6
> K6:=CyclotomicField(6);
> [Eigenvalues(ChangeRing(w,K6)) : w in liste6];
[
    {
        <zeta_6 - 1, 1>,
        <-zeta_6, 1>,
        <-zeta_6 + 1, 2>,
        <zeta_6, 2>
    },
    {
        <zeta_6 - 1, 1>,
        <-zeta_6, 1>,
        <-zeta_6 + 1, 2>,
        <zeta_6, 2>
    },
    {
        <1, 2>,
        <-zeta_6 + 1, 1>,
        <zeta_6, 1>,
        <-1, 2>
    },
    {
        <1, 2>,
        <-zeta_6 + 1, 1>,
        <zeta_6, 1>,
        <-1, 2>
    },
    {
        <-zeta_6 + 1, 1>,
        <zeta_6, 1>,
        <-1, 2>,
        <zeta_6 - 1, 1>,
        <-zeta_6, 1>
    },
    {
        <1, 2>,
        <-1, 2>,
        <zeta_6 - 1, 1>,
        <-zeta_6, 1>
    }
]
> ww6:=[w : w in conjDW |
  Dimension(Eigenspace(ChangeRing(w,K6),K6.1)) eq 2];
> w6:=ww6[1];
> IsConjugate(DW,w6,w6^-1);
false
> Xw6:=FixedPoints(w6) meet X;
> Xw6 eq Xw6 meet Scheme(P5,[f5,f9]);
true
> [Dimension(FixedPoints(w) meet X) : w in liste6 | (w in ww6) eq false];
[ -1, -1, -1, -1 ]
\end{verbatim}}
\end{quotation}
\end{comp}

\begin{comp}\label{comp:order 9 12}~
\begin{quotation}
{\small\begin{verbatim}
> liste9:=[w : w in conjDW | Order(w) eq 9];
> # liste9;
2
> [Eigenvalues(ChangeRing(w,CyclotomicField(9))) : w in liste9];
[
    {
        <-zeta_9^5 - zeta_9^2, 1>,
        <zeta_9^2, 1>,
        <-zeta_9^4 - zeta_9, 1>,
        <zeta_9, 1>,
        <zeta_9^4, 1>,
        <zeta_9^5, 1>
    },
    {
        <-zeta_9^5 - zeta_9^2, 1>,
        <zeta_9^2, 1>,
        <-zeta_9^4 - zeta_9, 1>,
        <zeta_9, 1>,
        <zeta_9^4, 1>,
        <zeta_9^5, 1>
    }
]
> w9:=liste9[1];
> IsConjugate(DW,w9,w9^-1);
false
> Xw9:=FixedPoints(w9) meet X;
> Dimension(Xw9);
0
> Degree(Xw9);
6
> liste12:=[w : w in conjDW | Order(w) eq 12];
> # liste12;
2
> [Eigenvalues(ChangeRing(w,CyclotomicField(12))) : w in liste12];
[
    {
        <zeta_12, 1>,
        <zeta_12^2 - 1, 1>,
        <-zeta_12^3 + zeta_12, 1>,
        <-zeta_12, 1>,
        <-zeta_12^2, 1>,
        <zeta_12^3 - zeta_12, 1>
    },
    {
        <zeta_12, 1>,
        <zeta_12^2 - 1, 1>,
        <-zeta_12^3 + zeta_12, 1>,
        <-zeta_12, 1>,
        <-zeta_12^2, 1>,
        <zeta_12^3 - zeta_12, 1>
    }
]
> w12:=liste12[1];
> IsConjugate(DW,w12,w12^-1);
false
> Xw12:=FixedPoints(w12) meet X;
> Dimension(Xw12);
0
> Degree(Xw12);
4
\end{verbatim}}
\end{quotation}
\end{comp}

\begin{comp}\label{comp:order}~
\begin{quotation}
{\small\begin{verbatim}
> Set([Order(w) : w in conjDW]);
{ 1, 2, 3, 4, 5, 6, 9, 12 }
\end{verbatim}}
\end{quotation}
\end{comp}

\bigskip

\subsection{Coordinates of singular points}

\begin{comp}\label{comp:p9p12}~
\begin{quotation}
{\small\begin{verbatim}
> K9:=CyclotomicField(9);
> v9:=Basis(Eigenspace(ChangeRing(Transpose(w9),K9),K9.1))[1];
> v9:=[v9[k] : k in [1..6]];
> p9:=[Evaluate(f9,v9),&* [Evaluate(r,v9) : r in roots]];
> p9[2]^2/p9[1]^8;
-2187/4096
>
> K12:=CyclotomicField(12);
> v12:=Basis(Eigenspace(ChangeRing(Transpose(w12),K12),K12.1))[1];
> v12:=[v12[k] : k in [1..6]];
> p12:=[Evaluate(f12,v12),&* [Evaluate(r,v12) : r in roots]];
> p12[2]^2/p12[1]^6;
-256/19683
\end{verbatim}}
\end{quotation}
\end{comp}

\begin{comp}\label{comp:p1}~
\begin{quotation}
{\small\begin{verbatim}
> Av2:=AffinePatch(Xv2,1);
> A3:=AffineSpace(Rationals(),3);
> AFF<[u]>:=AmbientSpace(Av2);
> phi:=map<Av2 -> A3 | [Evaluate(f,u cat [1]) : f in [f5^3,f9,f12]]>;
> Q19<r19>:=QuadraticField(19);
> test:=phi(Av2);
> pts:=Points(test,Q19);
> xi:=-1/3;
> p1:=[xi^5*pts[1][1],xi^3*pts[1][2],xi^4*pts[1][3]];
> pp1:=[xi^5*pts[2][1],xi^3*pts[2][2],xi^4*pts[2][3]];
> p1;pp1;
[ 1/9*(536*r19 + 2336), 1/9*(-60*r19 - 260), 1/3*(130*r19 + 565) ]
[ 1/9*(-536*r19 + 2336), 1/9*(60*r19 - 260), 1/3*(-130*r19 + 565) ]
\end{verbatim}}
\end{quotation}
\end{comp}

\begin{comp}\label{comp:p2}~
\begin{quotation}
{\small\begin{verbatim}
> AA:=AffinePatch(A,1);
> A3:=AffineSpace(Rationals(),3);
> AFF<[u]>:=AmbientSpace(AA);
> phi:=map<AA -> A3 | [Evaluate(f,u cat [1]) : f in [f5^3,f9,f12]]>;
> p2:=phi(AA);
> K<a,b,c>:=CoordinateRing(p2);
> IsField(K);
true
> pts:=Points(p2,K);
> p:=pts[1];
> xi:=-2^4/(3*5)*p[3]/p[1];
> [xi^5*p[1],xi^3*p[2],xi^4*p[3]];
[
    4/9,
    10/9,
    -5/12
]
\end{verbatim}}
\end{quotation}
\end{comp}

\subsection{Equation of ${\boldsymbol{\XC/W'}}$}

\begin{comp}\label{comp:p}~
\begin{quotation}
{\small\begin{verbatim}
> Scheme(Xv3,f5) eq Xv3;
true
> Set([Scheme(Xv3,r) eq Xv3 : r in roots]);
{ false, true }
> Av3:=AffinePatch(Xv3,1);
> A2:=AffineSpace(Rationals(),2);
> AFF<[u]>:=AmbientSpace(Av3);
> phi:=map<Av3 -> A2 | [Evaluate(f,u cat [1]) : f in [f9,f12]]>;
> test:=phi(Av3);
> K:=CoordinateRing(test);
> IsField(K);
true
> pts:=Points(test,K);
> p:=pts[1];
> xi:=(8/9)*p[2]/p[1];
> p:=[0,xi^-3*p[1],xi^-4*p[2],0];
\end{verbatim}}
\end{quotation}
\end{comp}

\begin{comp}\label{comp:r}~
\begin{quotation}
{\small\begin{verbatim}
> E:=FixedPoints(W.2) meet Scheme(X,x1);
> E:=Scheme(P5,MinimalBasis(E));
> E:=ReducedSubscheme(E);
> Dimension(E);
0
> Degree(E);
96
> IE:=IrreducibleComponents(E);
> # IE;
1
> IE:=[Scheme(P5,MinimalBasis(irr)) : irr in IE];
>
> AIE:=AffinePatch(IE[1],1);
> A3:=AffineSpace(Rationals(),3);
> AFF<[u]>:=AmbientSpace(AIE);
> phi:=map<AIE -> A3 | [Evaluate(f,u cat [1]) : f in [f5^3,f9,f12]]>;
> test:=phi(AIE);
> psi:=map<test-> A3 | [A3.3^5/A3.1^4,A3.2*A3.3^3/A3.1^3,A3.3^5/A3.1^4]>;
> pts:=Points(psi(test),Q6);
> r:=pts[1];
> r:=[r[1],r[2],r[3],0];
> rp:=pts[2];
> rp:=[rp[1],rp[2],rp[3],0];
\end{verbatim}}
\end{quotation}
\end{comp}

\begin{comp}\label{comp:equation}~
\begin{quotation}
{\small\begin{verbatim}
> K:=CompositeFields(Q6,Q19)[1];
> POLY<a,b,c,d,e>:=PolynomialRing(K,5);
>
> equation:=function(p) local res;
function>   res:= - p[4]^2 - 3^7/2^12*(p[2]^4-2^10/3^8*p[3]^3)^2
function>         - p[1]*p[2]*p[3]*(p[2]^4*a+p[3]^3*b)
function>         - p[1]^2*p[2]^2*p[3]^2*c - p[1]^3*p[2]^3*d - p[1]^4*p[3]*e;
function>   return res;
function> end function;
>
> eq1:=equation(p1);
> eqp1:=equation(pp1);
> eq2:=equation(p2);
> eqp:=equation(p);
> eqpp:=equation(pp);
>
> solutions:=Scheme(Spec(POLY),[eq1,eqp1,eq2,eqp,eqpp]);
> Points(solutions);
{@ (207/32, 800/729, 1375/81, -3125/864, -3125/108) @}
\end{verbatim}}
\end{quotation}
\end{comp}

\bigskip

\subsection{The affine chart ${\boldsymbol{\UC_5}}$}\label{sub:u5}
The presentation of the affine chart $\UC_5$ is much more difficult than
the one of $\UC_3$ or $\UC_4$, since
the singularity $\CM^3/\mub_5$ (with $\mub_5$ acting via
$\xi \cdot (y_3,y_4,j)=(\xi^3 y_3,\xi^4 y_4, \xi^2 j)$) is more
complicated than the $A_2$ or $A_3$ singularities involved in $\UC_3$ and $\UC_4$.

\def\cinq{{\SSS{5}}}

\bigskip

\begin{comp}\label{comp:u5}~
\begin{quotation}
{\small\begin{verbatim}
> K5<r>:=CyclotomicField(5);
> g:=DiagonalMatrix(K5,[r^3,r^4,r^2]);
> G:=MatrixGroup<3,K5 | g>;
> RG:=InvariantRing(G);
> PG<y3,y4,j>:=PolynomialRing(RG);
> RG`PrimaryInvariants:=[y3^5,y4^5,j^5];
> sec:=IrreducibleSecondaryInvariants(RG);
> # sec;
5
> sec;
[
    y3*j,
    y3^2*y4,
    y4^2*j,
    y3*y4^3,
    y4*j^3
]
> ALG<Y3,Y4,J,h1,h2,h3,h4,h5>:=Algebra(RG);
> rels:=Relations(RG);
> rels;
[
    -h1^2*h3 + h2*h5,
    Y4*h1 - h3*h4,
    J*h2 - h1^2*h5,
    -h1*h2^2 + Y3*h3,
    Y4*h2 - h4^2,
    J*h3 - h5^2,
    -h1^3*h2 + Y3*h5,
    J*h4 - h1*h3*h5,
    Y4*h1^2 - h2*h3^2,
    -h2^3 + Y3*h4,
    -h3^3 + Y4*h5,
    h1^2*h2*h3 - h1^3*h4,
    Y3*J - h1^5,
    -h1^2*h3^2 + h1*h4*h5,
    Y3*Y4 - h2^2*h4,
    Y4*J - h3^2*h5,
    Y3*h3^2 - h1^2*h2*h4
]
\end{verbatim}}
\end{quotation}
\end{comp}

Computation~\ref{comp:u5} shows that, if we set
$$Y_3=y_3^5,\quad Y_4=y_4^5,\quad J=j^5,$$
$$h_1=y_3j,\quad h_2=y_3^2 y_4,\quad h_3=y_4^2 j,\quad h_4=y_3 y_4^3
\quad\text{and}\quad h_5=y_4 j^3,$$
then the invariant ring $\CM[\CM^3]^{\mub_5}$ admits the following presentation:
\equat\label{eq:u5}
\begin{cases}
\text{Generators:} & Y_3, Y_4, J, h_1, h_2, h_3, h_4, h_5;\\
\text{Relations:} & \text{see the last command of Computation~\ref{comp:u5}}
\end{cases}
\endequat
If we want a description of $(\XC/W') \cap \UC_5$, we first multiply the equation~\eqref{eq:equation} by
$y_4j$, and we get that
$$h_5=\Phi_5(Y_3,Y_4,h_1,h_2,h_3,h_4)$$
for some polynomial $\Phi_5$ in six variables. If we multiply the equation~\eqref{eq:equation}
by $j^3$, we get that
$$J=\Psi_5(Y_3,Y_4,h_1,h_2,h_3,h_4,h_5)$$
for some polynomial $\Phi_5$ in seven variables.
If we denote by $\WC_5$ the closed subscheme of
$\UC_5$ defined by the equations
\equat\label{eq:h5J}
\begin{cases}
h_5=\Phi_5(Y_3,Y_4,h_1,h_2,h_3,h_4),\\
J=\Psi_5(Y_3,Y_4,h_1,h_2,h_3,h_4,h_5),
\end{cases}
\endequat
then $(\XC/W') \cap \UC_5$ is an irreducible
component of the reduced subscheme of $\WC_5$. The next computation determines
the equation of $(\XC/W') \cap \UC_5$ thanks to this remark, because $(\XC/W') \cap \UC_5$
contains more than only one singular point (it is encoded in the variable {\tt XWU5}):

\begin{comp}\label{comp:xdw dans u5}~
\begin{quotation}
{\small\begin{verbatim}
> AFFG:=Spec(ALG);
> U5:=Scheme(AFFG,rels);
> Phi5:= - 2187/4096*Y3*h1*h2 + 1/6*h1*h2*h4 - 207/32*Y3*h3 - 256/19683*Y4*h3
  - 800/729*h3*h4 - 1375/81*h1*h4 + 3125/864*h1*h2 + 3125/108*h3;
> Psi5:= - 2187/4096*Y3*h1^3 + 1/6*h2^2*h5 - 207/32*Y3*h5 - 256/19683*Y4*h5
  - 800/729*h4*h5 - 1375/81*h2*h5 + 3125/864*h1^3 + 3125/108*h5;
> W5:=Scheme(U5,[h5 - Phi5, J - Psi5]);
> W5:=ReducedSubscheme(W5);
> W5:=Scheme(AFFG,MinimalBasis(W5));
> IrrW5:=IrreducibleComponents(W5);
> IrrW5:=[Scheme(AFFG,MinimalBasis(i)) : i in IrrW5];
> [# MinimalBasis(i) : i in IrrW5];
[ 20, 7 ]
> test:=[i : i in IrrW5 | # MinimalBasis(i) lt 10];
> [Degree(ReducedSubscheme(SingularSubscheme(i))) : i in test];
[ 1 ]
> XWU5:=[i : i in IrrW5 | # MinimalBasis(i) gt 10];
> XWU5:=XWU5[1];
\end{verbatim}}
\end{quotation}
\end{comp}

In order to simplify the equations of $(\XC/W') \cap \UC_5$, we can use the equations~\eqref{eq:h5J}
to eliminate the variables $h_5$ and $J$, so that we can embed $(\XC/W') \cap \UC_5$ into
$\AM^6(\CM)$ instead of $\AM^8(\CM)$. This is done as follows in {\sc Magma}:

\begin{comp}\label{comp:xwu5 dans a6}~
\begin{quotation}
{\small\begin{verbatim}
> A6<Y3,Y4,h1,h2,h3,h4>:=AffineSpace(Rationals(),6);
> h5:= - 2187/4096*Y3*h1*h2 + 1/6*h1*h2*h4 - 207/32*Y3*h3
> - 256/19683*Y4*h3 - 800/729*h3*h4 - 1375/81*h1*h4
> + 3125/864*h1*h2 + 3125/108*h3;
> J:= - 2187/4096*Y3*h1^3 + 1/6*h2^2*h5 - 207/32*Y3*h5 - 256/19683*Y4*h5
> - 800/729*h4*h5 - 1375/81*h2*h5 + 3125/864*h1^3 + 3125/108*h5;
> XWU5:=Scheme(A6, [Evaluate(f,[Y3,Y4,J,h1,h2,h3,h4,h5]) :
  f in MinimalBasis(XWU5)]);
> XWU5:=Scheme(A6,MinimalBasis(XWU5));
> XWU5;
Scheme over Rational Field defined by
-h1*h4 + h2*h3,
Y4*h2 - h4^2,
Y4*h1 - h3*h4,
Y3^2 + 2944/243*Y3*h2 - 2048/6561*Y3*h4 - 400000/59049*Y3
+ 4096/2187*h1^2 + 5632000/177147*h2^2 + 3276800/1594323*h2*h4
- 3200000/59049*h2 + 1048576/43046721*h4^2,
-Y3*h4 + h2^3,
-Y3*h3 + h1*h2^2,
Y3*h2*h4 + 1048576/43046721*Y4^2 + 3276800/1594323*Y4*h4
- 3200000/59049*Y4 + 2944/243*h2^2*h4 - 2048/6561*h2*h4^2
- 400000/59049*h2*h4 + 4096/2187*h3^2 + 5632000/177147*h4^2,
Y3*h2^2 + 2944/243*Y3*h4 + 1048576/43046721*Y4*h4 + 4096/2187*h1*h3
- 2048/6561*h2^2*h4 - 400000/59049*h2^2 + 5632000/177147*h2*h4
+ 3276800/1594323*h4^2 - 3200000/59049*h4,
Y3*Y4 - h2^2*h4
\end{verbatim}}
\end{quotation}
\end{comp}

%
%


\section{Magma computations for Section~\ref{sec:action}}

\medskip

We first compute the character table of $W$, the fake degrees and the image $\DBCC(W)$
of the map $\DBrm : \Irr(W) \longto \NM \times \NM$ defined in~\eqref{eq:phidb}.
Note that $\DBCC(W)$ is denoted by {\tt DBW} in the computation below.

\bigskip

\begin{comp}\label{comp:fake}~
\begin{quotation}
{\small\begin{verbatim}
TW:=CharacterTable(W);
conjW:=ConjugacyClasses(W);
# conjW;
C<t>:=FunctionField(Q);
poincare:=(1-t^2)*(1-t^5)*(1-t^6)*(1-t^8)*(1-t^9)*(1-t^12);
fake:=function(chi) local c,res;
  res:=&+ [c[2]*(Q ! chi(c[3]^-1))/Determinant(IdentityMatrix(C,6)
           -t*ChangeRing(c[3],C)) : c in conjW];
  return poincare*res/Order(W);
end function;
b:=function(chi) return Valuation(Numerator(fake(chi)));end function;
DBW:=[[TW[k](W.0),b(TW[k])] : k in [1..# conjW]];
\end{verbatim}}
\end{quotation}
\end{comp}

\bigskip

\begin{comp}\label{comp:ew}~
\begin{quotation}
{\small\begin{verbatim}
> # Set(DBW) eq # DBW;
true
> # DBW; <<i : i in j> : j in DBW>;
25
<<1, 0>, <1, 36>, <6, 1>, <6, 25>, <10, 9>, <15, 17>, <15, 4>,
<15, 16>, <15, 5>, <20, 20>, <20, 10>, <20, 2>, <24, 6>, <24, 12>,
<30, 3>, <30, 15>, <60, 11>, <60, 8>, <60, 5>, <64, 13>, <64, 4>,
<80, 7>, <81, 6>, <81, 10>, <90, 8>>
\end{verbatim}}
\end{quotation}
\end{comp}

\bigskip

\begin{comp}\label{comp:fixed dim}~
\begin{quotation}
{\small\begin{verbatim}
> fixednonfinite:=[k : k in [1..# conjW] |
  Dimension(FixedPoints(conjW[k][3]) meet X) ge 1];
> fixednonfinite;
[ 1, 2, 6 ]
> [Order(conjW[k][3]) : k in fixednonfinite];
[ 1, 2, 3 ]
> [Eigenvalues(ChangeRing(conjW[k][3],CyclotomicField(3))) :
  k in fixednonfinite];
[
    {
        <1, 6>
    },
    {
        <-1, 1>,
        <1, 5>
    },
    {
        <zeta_3, 3>,
        <-zeta_3 - 1, 3>
    }
]
\end{verbatim}}
\end{quotation}
\end{comp}

\bigskip

\begin{comp}\label{comp:chix}~
\begin{quotation}
{\small\begin{verbatim}
> chix:=[0 : k in [1..# conjW]];
> for k in [1..# conjW] do
for>   if k in fixedcurves then
for|if>     if Order(conjW[k][3]) eq 1 then chix[k]:=9504;
for|if|if>       elif Order(conjW[k][3]) eq 2 then chix[k]:=-1056;
for|if|if>       else chix[k]:=-36;
for|if|if>     end if;
for|if>   else
for|if>     w:=conjW[k][3];
for|if>     chix[k]:=Degree(FixedPoints(w) meet X);
for|if>   end if;
for> end for;
> CHAR:=CharacterRing(W);
> mult:=Decomposition(TW,CHAR ! chix);
> for k in [1..# conjW] do
for>   if mult[k] ne 0 then
for|if>     if mult[k] eq 1 then
for|if>       printf "phi_{ %o , %o} + ",DBW[k][1],DBW[k][2];
for|if|if>     else printf "%o \\phi_{%o,%o} + ",mult[k],DBW[k][1],DBW[k][2];
for|if|if>     end if;
for|if>   end if;
for> end for;
3 \phi_{1,0} + 3 \phi_{1,36} + 8 \phi_{6,25} + 2 \phi_{10,9}
+ 7 \phi_{15,17} + phi_{ 15 , 4} + 9 \phi_{15,16} + phi_{ 15 , 5}
+ 14 \phi_{20,20} + 4 \phi_{20,10} + 2 \phi_{24,6} + 8 \phi_{24,12}
+ 14 \phi_{30,15} + 18 \phi_{60,11} + 12 \phi_{60,8} + 4 \phi_{60,5}
+ 26 \phi_{64,13} + 2 \phi_{64,4} + 12 \phi_{80,7} + 7 \phi_{81,6}
+ 21 \phi_{81,10} + 12 \phi_{90,8} +
\end{verbatim}}
\end{quotation}
\end{comp}

\bigskip

\section{Proof of Theorem~\ref{theo:elliptic}(a) and~(b)}\label{appendix:open}

\medskip

\subsection{Extension of ${\boldsymbol{\ph \circ \rho}}$}
Let $\rhoh : \XCh \to \XC/W'$ denote the blow-up of $\XC/W'$ at $p_5$. Since $p_5$ is an $A_4$-singularity,
$\XCh$ contains a unique singular point above $p_5$, which will be denoted by $\phat_5$. It
is an $A_2$-singularity of $\XCh$ and we denote by $\pih : \hat{\XCh} \to \XCh$
the blow-up of $\XCh$ at $\phat_5$ and we set $\hat{\rhoh} = \rhoh \circ \pih$.
Since $\XCt$ is obtained from $\XC/W'$ by successive blow-ups of the singular
points, there is a morphism $\XCt \to \hat{\XCh}$. Therefore,
in order to prove that $\ph \circ \r$ extends to $\XCt$,
we only need to prove that the map
$$\ph \circ \hat{\rhoh} : \hat{\XCh} \setminus \hat{\rhoh}^{-1}(p_5) \to \PM^1(\CM)$$
extends to a morphism $\hat{\XCh} \to \PM^1(\CM)$.

For this, it is sufficient to work in the affine chart $\UC_5$. Let $\ph_5$ denote the
restriction of $\ph$ to $((\XC/W') \cap \UC_5) \setminus\{p_5\}$. Keeping the notation
of~\S\ref{sub:affine charts} and~\S\ref{sub:u5}, $\ph_5$ is defined by
$$\ph_5(Y_3,Y_4,h_1,h_2,h_3,h_4)=
\begin{cases}
[Y_3:h_4] & \text{if $(Y_3,h_4) \neq (0,0)$,}\\
[h_2^2 : Y_4] & \text{if $(h_2,Y_4) \neq (0,0)$.}
\end{cases}$$
Note that if $(Y_3,Y_4,h_1,h_2,h_3,h_4) \in (\XC/W') \cap \UC_5$ is such that
$Y_3=Y_4=h_2=h_4=0$, then $h_1=h_3=0$. So $\ph_5$ is indeed well-defined by
the above formula on
$((\XC/W') \cap \UC_5) \setminus\{p_5\}$.

We denote by $\XCh_5$ the inverse image of $(\XC/W') \cap \UC_5$ in $\XCh$ (so that $\XCh_5$ is the
blow-up of $(\XC/W') \cap \UC_5$ at $p_5$.
So $\XCh_5$ is embedded in $\AM^6(\CM) \times \PM^5(\CM)$. If $t \in \{Y_3,Y_4,h_1,h_2,h_3,h_4\}$,
we denote by $t'$ the corresponding variable in $\PM^5(\CM)$ and by $\XCh_5^{(t)}$ the open affine chart
of $\XCh_5$ defined by $t' \neq 0$. Then:

\begin{lem}\label{lem:phat5}
The point $\phat_5$ belongs to the affine chart
$\XCh_5^{(h_1)}$ and does no belong to the others.
\end{lem}

\bigskip

\begin{proof}
This follows from the following computation:

\begin{comp}\label{comp:phat5}~
\begin{quotation}
{\small\begin{verbatim}
> Xhat5:=LocalBlowUp(XWU5,Scheme(A6,[A6.k : k in [1..6]]));
> [Dimension(SingularSubscheme(Xhat5[k][1])
meet Scheme(XWU5,[Y3,Y4,h1,h2,h3,h4]) @@ Xhat5[k][2]) : k in [1..6]];
[ -1, -1, 0, -1, -1, -1 ]
\end{verbatim}}
\end{quotation}
\end{comp}
\end{proof}

\bigskip

To prove that $\ph \circ \hat{\rhoh}$ extends to $\hat{\XCh}$, we first prove the following lemma:

\bigskip

\begin{lem}\label{lem:first extension}
The morphism $\ph \circ \rhoh : \XCh \setminus \rhoh^{-1}(p_5) \to \PM^1(\CM)$ extends to
a morphism $\XCh \setminus \{\phat_5\} \to \PM^1(\CM)$.
\end{lem}

\bigskip

\begin{proof}
If $(Y_3,h_4) \neq (0,0)$, then $[Y_3 : h_4]=[Y_3': h_4']$ and so the map $\ph_5 \circ \rhoh$
extends to the affine charts $\XCh_5^{(Y_3)}$ and $\XCh_5^{(h_4)}$.
Similarly, if $Y_4 \neq 0$, then $[h_2^2 : Y_4]=[h_2h_2' : Y_4']$ and so
$\ph_5 \circ \rhoh$
extends to the affine chart $\XCh_5^{(Y_4)}$.

Now, it follows from Computation~\ref{comp:blowup 1} below
that the two affine charts $\XCh_5^{(h_2)}$ and $\XCh_5^{(h_3)}$
can be embedded in $\AM^4(\CM)$, and that maps $\XCh_5^{(h_2)} \to (\XC/W') \cap \UC_5$
and $\XCh_5^{(h_3)} \to (\XC/W') \cap \UC_5$ are given by
$$\fonctio{\XCh_5^{(h_2)}}{(\XC/W') \cap \UC_5}{(a,b,c,d)}{(a^2d,ad^3,abd,ad,acd,ad^2)}$$
$$\fonctio{\XCh_5^{(h_3)}}{(\XC/W') \cap \UC_5}{(a,b,c,d)}{(ab,ac,ad,acd^2,a,acd).}\leqno{\text{and}}$$
So the restriction of $\ph_5 \circ \rhoh$ to $\XCh_5^{(h_2)} \setminus \rhoh^{-1}(p_5)$
(resp. $\XCh_5^{(h_3)} \setminus \rhoh^{-1}(p_5)$) is given
by $(a,b,c,d) \mapsto [a^2d:ad^2]$ (resp. $(a,b,c,d) \mapsto [a^2c^2d^4 : ac]$).
Now, $(a,b,c,d) \mapsto [acd^4 : 1]$
is well-defined on $\XCh_5^{(h_2)}$ and extends $\ph_5 \circ \rhoh$.
Moreover, $(a,b,c,d) \mapsto [a:d]$ is well-defined on $\BC$ (thanks to the
command {\tt Dimension(Scheme(Xhat5[4][1],[a[1],a[4]]));} in Computation~\ref{comp:blowup 1}
below, which shows that the subvariety of $\XCh_5^{(h_3)}$ defined by $a=d=0$ is empty)
and extends $\ph_5 \circ \rhoh$.

\begin{comp}\label{comp:blowup 1}~
\begin{quotation}
{\small\begin{verbatim}
> Xhat5:=LocalBlowUp(XWU5,Scheme(A6,[A6.k : k in [1..6]]));
>
> AFF4<[a]>:=AmbientSpace(Xhat5[4][1]);
> rhoh4:=Xhat5[4][2];
> DefiningEquations(rhoh4);
[
    a[1]^2*a[4],
    a[1]*a[4]^3,
    a[1]*a[2]*a[4],
    a[1]*a[4],
    a[1]*a[3]*a[4],
    a[1]*a[4]^2
]
>
> Dimension(Scheme(Xhat5[4][1],[a[1],a[4]]));
-1
>
> AFF5<[a]>:=AmbientSpace(Xhat5[5][1]);
> rhoh5:=Xhat5[5][2];
> DefiningEquations(rhoh5);
[
    a[1]*a[2],
    a[1]*a[3],
    a[1]*a[4],
    a[1]*a[3]*a[4]^2,
    a[1],
    a[1]*a[3]*a[4]
]
\end{verbatim}}
\end{quotation}
\end{comp}

\bigskip

Therefore, it remains to understand what happens on the open affine subset of $\XCh_5^{(h_1)}$.
Then it follows from Computation~\ref{comp:blowup 2} below that
\equat\label{eq:XCh5h1}
\XCh_5^{(h_1)}=\{(a,b,c,d) \in \AM^4(\CM)~|~ac^2=bd \qquad\text{and} \hskip5cm
\endequat
$$\hskip1cm abcd - \frac{512}{6561}bd^3 - \frac{6561}{2048} ab^2 - \frac{621}{16} abc -
\frac{1600}{243} bd^2 - \frac{2750}{27} bd - 6 a + \frac{3125}{144} b + \frac{3125}{18}c=0\}$$
and that $\rhoh(a,b,c,d)=(ab,acd^2,a,ac,ad,acd)$.

\bigskip

\begin{comp}\label{comp:blowup 2}~
\begin{quotation}
{\small\begin{verbatim}
> AFF3<[a]>:=AmbientSpace(Xhat5[3][1]);
> rhoh3:=Xhat5[3][2];
> MinimalBasis(Xhat5[3][1]);
[
    a[1]*a[3]^2 - a[2]*a[4],
    a[1]*a[2]*a[3]*a[4] - 512/6561*a[2]*a[4]^3 - 6561/2048*a[1]*a[2]^2
    - 621/16*a[1]*a[2]*a[3] - 1600/243*a[2]*a[4]^2 - 2750/27*a[2]*a[4]
    - 6*a[1] + 3125/144*a[2] + 3125/18*a[3]
]
> DefiningEquations(rhoh3);
[
    a[1]*a[2],
    a[1]*a[3]*a[4]^2,
    a[1],
    a[1]*a[3],
    a[1]*a[4],
    a[1]*a[3]*a[4]
]
> MinimalBasis(ReducedSubscheme(SingularSubscheme(Xhat5[3][1])
  meet Scheme(XWU5,[Y3,Y4,h1,h2,h3,h4]) @@ Xhat5[3][2]));
[
    a[4],
    a[3],
    a[2],
    a[1]
]
\end{verbatim}}
\end{quotation}
\end{comp}

\bigskip

Note also that the last command of the Computation~\ref{comp:blowup 1} above shows that $\phat_5$
corresponds to the points $(0,0,0,0)$ in $\XCh_5^{(h_1)} \subset \AM^4(\CM)$.

In particular, the restriction of $\ph \circ \rhoh$ to $\XCh_5^{(h_1)} \setminus \rhoh^{-1}(p_5)$
is given by
$$(a,b,c,d) \longmapsto
\begin{cases}
[ab : acd] & \text{if $(ab,acd) \neq (0,0)$,}\\
[a^2c^2 : acd^2] & \text{if $(a^2c^2 , acd^2) \neq (0,0)$.}
\end{cases}$$
But this can be extended to a morphism $\phh : \XCh_5^{(h_1)} \setminus \{(0,0,0,0)\} \to \PM^1(\CM)$
by the formula
$$(a,b,c,d) \longmapsto
\begin{cases}
[b : cd] & \text{if $(b,cd) \neq (0,0)$,}\\
[ac : d^2] & \text{if $(ac , d^2) \neq (0,0)$.}
\end{cases}$$
Indeed, if $(a,b,c,d) \in \XCh^{(h_1)}$ satisfies $b=d=ac=0$, then $a=b=c=d=0$ (see the equations
of $\XCh_5^{(h_1)}$).
\end{proof}

\bigskip

Therefore, it remains to prove that
$\phh \circ \pih : \DCh \setminus \pih^{-1}(\phat_5) \to \PM^1(\CM)$
can be extended to $\DCh$, where $\DCh = \pih^{-1}(\XCh_5^{(h_1)})$ is an open subset of $\hat{\XCh}$.
Note that $\DCh$ is also the blow-up of $\XCh_5^{(h_1)}$ at $\phat$.
The Computation~\ref{comp:blowup h1} below shows that
$\DCh$ is embedded in $\AM^4(\CM) \times \PM^2(\CM)$ and is the variety of points
$((a,b,c,d),[b':c':d']) \in \AM^4(\CM) \times \PM^2(\CM)$ such that\footnote{A priori, $\DCh$ should be embedded
in $\AM^4(\CM) \times \PM^3(\CM)$, but the second equation giving $\XCh_5^{(h_1)}$ in~\eqref{eq:XCh5h1} has a
non-trivial hommogeneous component of degree $1$, which allows to eliminate one variable.}
\equat\label{eq:Dhat}
\begin{cases}
bc' = cb',\\
cd' = dc',\\
bd' = db',\\
ac^{\prime 2}= d'b',\\
acc' = db',\\
ac^2 = bd,\\
abcd - \frac{512}{6561}bd^3 - \frac{6561}{2048} ab^2 - \frac{621}{16} abc -
\frac{1600}{243} bd^2 - \frac{2750}{27} bd - 6 a + \frac{3125}{144} b + \frac{3125}{18}c=0.
\end{cases}
\endequat

\begin{comp}\label{comp:blowup h1}~
\begin{quotation}
{\small\begin{verbatim}
> Dhat:=Blowup(Xhat5[3][1],Scheme(Xhat5[3][1],a));
> A4P2<a,b,c,d,dpr,cpr,bpr>:=AmbientSpace(Dhat);
> A4P2;
Projective Space of dimension 6 over Rational Field
Variables: a, b, c, d, dpr, cpr, bpr
The grading is:
    0, 0, 0, 0, 1, 1, 1
> MinimalBasis(Dhat);
[
    b*cpr - c*bpr,
    c*dpr - d*cpr,
    b*dpr - d*bpr,
    -a*cpr^2 + dpr*bpr,
    -a*c*cpr + d*bpr,
    -a*c^2 + b*d,
    144/3125*a*b*c*d - 1024/3375*a*c^2*d - 8192/2278125*b*d^3
      - 59049/400000*a*b^2 - 5589/3125*a*b*c - 352/75*a*c^2
      - 864/3125*a + b + 8*c
]
\end{verbatim}}
\end{quotation}
\end{comp}
If $(a,b,c,d,b',c',d') \in \CM^7$, we set
\eqna
\a'(a,b,c,d,b',c',d')&=&\bigl(abcd' - \frac{512}{6561}bd^2d' - \frac{6561}{2048} abb' - \frac{621}{16} abc'\\
&& - \frac{1600}{243} bdd' - \frac{2750}{27} bd' + \frac{3125}{144} b' + \frac{3125}{18}c'\bigr)/6.
\endeqna
Then, if $((a,b,c,d),[b':c':d']) \in \DCh$, it is imediately checked from the equations~\eqref{eq:Dhat}
that
\equat\label{eq:facile bl}
\begin{cases}
\a'(a,b,c,d,b',c',d') b = a b',\\
\a'(a,b,c,d,b',c',d') c = a c',\\
\a'(a,b,c,d,b',c',d') d = a d',\\
\end{cases}
\endequat
Therefore, if $(b,cd) \neq (0,0)$, then $[b : cd]=[b': cd']$ and, if $(ac,d^2) \neq (0,0)$, then
$[ac : d^2] = [\a'(a,b,c,d,b',c',d')c' : d^{\prime 2}]$. But if
$b'=d'=\a'(a,b,c,d,b',c',d')c'=0$, then $b'=c'=d'=0$, which is impossible. Therefore,
the formula
$$
\hat{\phh}((a,b,c,d),[b',c',d']) =
\begin{cases}
[b' : cd'] & \text{if $(b',cd') \neq (0,0)$,}\\
[\a'(a,b,c,d,b',c',d')c' : d^{\prime 2}] & \text{if $(\a'(a,b,c,d,b',c',d')c' , d^{\prime 2})\neq(0,0)$,}\\
\end{cases}
$$
gives a well-defined morphism $\DCh \to \PM^1(\CM)$ which extends $\phh \circ \pih$.

This shows that the morphism $\pht : \XCt \to \PM^1(\CM)$ expected by Theorem~\ref{theo:elliptic}(a)
is well-defined.

\bigskip

\subsection{Elliptic fibration}\label{sub:theo-a}
By the work done in the previous subsection, it remains to show that at least
one fiber of $\pht$ is a smooth elliptic curve. For this, let $\SC$ be the set of singular
points of $\XC/W'$ different from $p_5$ and let $\SC' = \ph(\SC) \subset \PM^1(\CM)$.
Note that $\SC$ and $\SC'$ are finite.

Among the four rational curves $\D_5^k$, $1 \le k \le 4$, some of them are mapped to a
point under $\pht$ and some of them are mapped to the whole $\PM^1(\CM)$. We denote by
$\SC''$ the set of points which are image of the rational curves contracted by $\pht$.
We set $U=\PM^1(\CM) \setminus (\SC' \cup \SC'')$. We denote by $U'$ the open subset
of $\PM^1(\CM)$ such that the restriction of $\pht$ to $\pht^{-1}(U') \to U'$ is smooth.
Finally, we set $U^+=U \cap U'$.

Let $x \in U^+$. Since $[0:1]$ and $[1:0]$ belong to $\SC'$, we can write $x=[1:b]$
for some non-zero complex number $b$. Since $[1:b] \not\in \SC''$, $\ph^{-1}([1:b])$ meets $\UC_3$
and $\ph^{-1}([1:b]) \cap \UC_3$ is a dense open subset of $\pht^{-1}([1:b])$ meets $\UC_3$. But it follows
from~\eqref{eq:u3} that
$$
\begin{array}{l}
\ph^{-1}([1:b]) \cap \UC_3 \simeq \{(c_\trois,j) \in \AM^3(\CM)~|~\\
~\\
\hskip2cm j^2= - 3(\frac{27}{64}-\frac{16}{243}\, b)^2
-
\frac{207}{32} c_\trois - \frac{800}{729} b c_\trois - \frac{1375}{81} c_\trois^2
+ \frac{3125}{864\, b} c_\trois^3 + \frac{3125}{108\, b} c_\trois^4\}.
\end{array}$$
This equation shows that $\ph^{-1}([1:b]) \cap \UC_3$ has genus $1$, and
so $\ph^{-1}([1:b])$ is a smooth elliptic curve.

This shows that the morphism $\pht : \XCt \to \PM^1(\CM)$ is an elliptic fibration.

\bigskip

\subsection{Local study at ${\boldsymbol{p_9}}$ and ${\boldsymbol{p_{12}}}$}
Let us first study the situation at $p_9$. For this, we work in the affine chart $\UC_3$
and use the equations~\eqref{eq:u3}. In this chart, which is a closed subvariety of $\AM^4(\CM)$, the point $p_9$
has coordinates $(0,0,0,27\sqrt{-3}/64)$. Moving this point to the origin by working with the coordinate
$j_\trois=j-27\sqrt{-3}/64$, we get from~\eqref{eq:u3} that
$$
\begin{array}{l}
(\XC/W') \cap \UC_3 =
\{(a_\trois,b_\trois,c_\trois,j_\trois) \in \AM^4(\CM)~|~c_\trois^3=a_\trois b_\trois
~\text{and}~ \petitespace\\
\hskip1.5cm \petitespace (j_\trois+\sqrt{-3}\frac{27}{64})^2= - 3(\frac{27}{64}-\frac{16}{243}\, b_\trois)^2 -
\frac{207}{32} c_\trois - \frac{800}{729} b_\trois c_\trois - \frac{1375}{81} c_\trois^2
+ \frac{3125}{864} a_\trois + \frac{3125}{108} a_\trois c_\trois \},
\end{array}$$
$$
\begin{array}{l}
\CC_4=\{(a_\trois,b_\trois,c_\trois,j_\trois) \in \AM^4(\CM)~|~b_\trois=c_\trois=0
~\text{and}~j_\trois^2 + \sqrt{-3}\frac{27}{32}\, j_\trois = \frac{3125}{864} a_\trois\},
\end{array}$$
$$\begin{array}{l}
\CC_5^+=\{(a_\trois,b_\trois,c_\trois,j_\trois) \in \AM^4(\CM)~|~a_\trois=c_\trois=0
~\text{and}~j_\trois = -\sqrt{-3}\frac{16}{243}\, b_\trois \}
\end{array}$$
$$p_9=(0,0,0,0) \in \CC_4 \cap \CC_5^+.\leqno{\text{and}}$$
Working in the completed local ring of $\UC_3$ at $p_9$, we can replace the variable $j_\trois$
by $j_\trois'=j_\trois^2 + \sqrt{-3}\frac{27}{32}\, j_\trois$, so that, in $\hat{\OC}_{\XC/W',p_9}$,
$j_\trois'$ can be expressed in terms of the variables $a_\trois$, $b_\trois$, $c_\trois$.
In other words,
$$\hat{\OC}_{\XC/W',p_9}  \simeq \CM[[a_\trois,b_\trois,c_\trois]]/\langle c_\trois^3 - a_\trois b_\trois\rangle$$
and
$$\hat{\OC}_{\CC_4,p_9} \simeq \hat{\OC}_{\XC/W',p_9}/\langle b_\trois,c_\trois \rangle
\qquad \text{and}\qquad \hat{\OC}_{\CC_5^+,p_9}
\simeq \hat{\OC}_{\XC/W',p_9}/\langle a_\trois,c_\trois \rangle.$$
The computation of the intersections of $\CCt^+_5$, $\CCt_4$, $\D_9^1$ and $\D_9^2$ can be done
with these completed local rings, and we get that, after exchanging $\D_9^1$ and $\D_9^2$ if necessary,
the following result holds:

\bigskip

\begin{lem}\label{lem:inter-p9}
The curve $\CCt^+_5$ (resp. $\CCt_4$) intersects $\D_9^1$ (resp. $\D_9^2$) transversely at
one point and does not intersect $\D_9^2$ (resp. $\D_9^1$).
\end{lem}

\bigskip

The situation at $p_{12}$ is very similar to the situation at $p_9$. Indeed, we work in the affine
chart $\UC_4$ and use the equations~\eqref{eq:u4} and we obtain that
$$\hat{\OC}_{\XC/W',p_{12}}  \simeq \CM[[a_\quatre,b_\quatre,c_\quatre]]/\langle c_\quatre^4 -
a_\quatre b_\quatre\rangle$$
and
$$\hat{\OC}_{\CC_3,p_{12}} \simeq \hat{\OC}_{\XC/W',p_{12}}/\langle b_\quatre,c_\quatre \rangle
\qquad \text{and}\qquad \hat{\OC}_{\CC_5^+,p_{12}}
\simeq \hat{\OC}_{\XC/W',p_{12}}/\langle a_\quatre,c_\quatre \rangle.$$
Therefore, after exchanging $\D_{12}^1$ and $\D_{12}^3$ if necessary,
the following result holds:

\bigskip

\begin{lem}\label{lem:inter-p12}
The curve $\CCt^+_5$ (resp. $\CCt_3$) intersect $\D_{12}^1$ (resp. $\D_{12}^3$) transversely at
one point and does not intersect $\D_{12}^2 \cup \D_{12}^3$ (resp. $\D_{12}^1 \cup \D_{12}^2$).
\end{lem}

\bigskip

Finally, since $\CC_5^+$ and $\CC_5^-$ intersect transversely at a smooth point of $\XC/W$, we get:

\bigskip

\begin{lem}\label{lem:inter c5}
The curves $\CCt^+_5$ and $\CCt^-_5$ intersect transversely at only one
point.
\end{lem}

\bigskip

Using the automorphism $\s$, Lemmas~\ref{lem:inter-p9},~\ref{lem:inter-p12} and~\ref{lem:inter c5}
show that the intersection
graph of the $14$ smooth rational curves $\CCt_3$, $\CCt_4$, $\CCt_5^+$, $\CCt_5^-$, $(\D_9^k)_{1 \le k \le 2}$,
$\lexp{\a}{\D_9^k})_{1 \le k \le 2}$, $(\D_{12}^k)_{1 \le k \le 3}$ and $(\lexp{\s}{\D_{12}^k})_{1 \le k \le 3}$
is given by:
$$
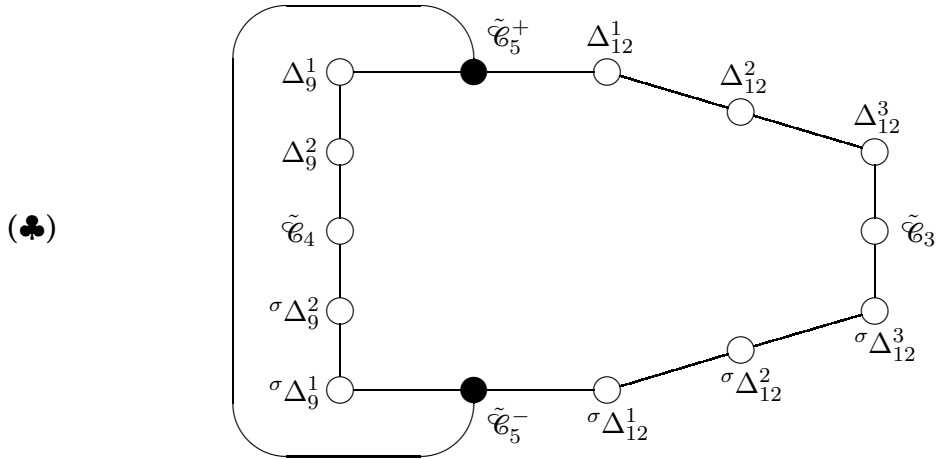
\begin{picture}(250,170)
\put( 25,140){\circle{10}}\put(03,136){$\D_9^1$}
\put( 75,140){\circle*{10}}\put(80,150){$\CCt_5^+$}
\put(125,140){\circle{10}}\put(117,150){$\D_{12}^1$}
\put(175,125){\circle{10}}\put(167,135){$\D_{12}^2$}
\put(225,110){\circle{10}}\put(217,120){$\D_{12}^3$}
\put( 25,110){\circle{10}}\put(3,106){$\D_9^2$}
\put( 25, 80){\circle{10}}\put(3,76){$\CCt_4$}
\put(225, 80){\circle{10}}\put(235,76){$\CCt_3$}
\put( 25, 50){\circle{10}}\put(-2,46){$\lexp{\s}{\D_{9}^2}$}
\put( 25, 20){\circle{10}}\put(-2,16){$\lexp{\s}{\D_{9}^1}$}
\put( 75, 20){\circle*{10}}\put(80,3.3){$\CCt_5^-$}
\put(125, 20){\circle{10}}\put(117,4){$\lexp{\s}{\D_{12}^1}$}
\put(175, 35){\circle{10}}\put(167,19){$\lexp{\s}{\D_{12}^2}$}
\put(225, 50){\circle{10}}\put(217,34){$\lexp{\s}{\D_{12}^3}$}

\put( 30, 140){\line(1,0){40}}
\put( 25, 135){\line(0,-1){20}}
\put( 80, 140){\line(1,0){40}}
\put( 25, 105){\line(0,-1){20}}
\put( 25, 75){\line(0,-1){20}}
\put( 25, 45){\line(0,-1){20}}
\put(225, 105){\line(0,-1){20}}
\put(225, 75){\line(0,-1){20}}
\put( 30, 20){\line(1,0){40}}
\put( 80, 20){\line(1,0){40}}


\qbezier(129.789,138.563)(150,132.5)(170.211,126.437)
\qbezier(179.789,123.563)(200,117.5)(220.211,111.437)
\qbezier(129.789,21.437)(150,27.5)(170.211,33.563)
\qbezier(179.789,36.563)(200,42.5)(220.211,48.437)

\put(30,140){\oval(90,50)[t]}
\put(30,20){\oval(90,50)[b]}
\put(-15,140){\line(0,-1){120}}

\put(-100,78){${\boldsymbol{(\clubsuit)}}$}
\end{picture}
$$
\smallskip

\noindent
This shows part of Theorem~\ref{theo:elliptic}(b).

\bigskip

\subsection{Singular fibers}
The study of singular fibers will be done in several steps:

\bigskip

\subsubsection{Singular fiber above $[0:1]$}
The fiber $\pht^{-1}([0:1])$ contains $\D_{12}^1$, $\D_{12}^2$, $\D_{12}^3$,
$\lexp{\s}{\D_{12}^1}$, $\lexp{\s}{\D_{12}^2}$, $\lexp{\s}{\D_{12}^3}$,
$\CCt_3$ and maybe some of the curves $\D_5^k$. Since the curves $\D_5^k$ do not intersect
the curves $\D_{12}^l$ and $\lexp{\s}{\D_{12}^m}$,
the Kodaira classification of singular fibers (together with the graph~$(\clubsuit)$)
shows that $\pht^{-1}([0:1])$ contains
exactly one of the curves $\D_5^k$ and is a singular fiber of type $\Erm_7$.

\bigskip

\subsubsection{Singular fiber above $[1:0]$}
The fiber $\pht^{-1}([1:0])$ contains $\D_9^1$, $\D_9^2$, $\lexp{\s}{\D_9^1}$, $\lexp{\s}{\D_9^2}$,
$\CCt_4$ and maybe some of the curves $\D_5^k$. Since the curves $\D_5^k$ do not intersect
the curves $\D_9^l$ and $\lexp{\s}{\D_9^m}$,
the Kodaira classification of singular fibers (together with the graph~$(\clubsuit)$)
shows that $\pht^{-1}([1:0])$ contains
exactly two of the curves $\D_5^k$ and is a singular fiber of type $\Erm_6$.

\bigskip

\subsubsection{Singular fibers above $\ph(p_1)$, $\ph(p_1')$ and $\ph(p_2')$}
Let $\EC_1=\pht^{-1}(\ph(p_1))$, $\EC_1'=\pht^{-1}(\ph(p_1'))$ and $\EC_2=\pht^{-1}(\ph(p_2))$.
Since $\EC_1$ contains $\D_1$ and at least one other irreducible curve,
the Kodaira classification of singular fibers shows that the Euler characteristic $\chib(\EC_1)$
of $\EC_1$ is $\ge 2$, and is equal to $2$ if and only if it is of type $\Irm_2$
in Kodaira's notation. The same hold for $\EC_1'$. Also,
$\EC_2$ contains $\D_2^1 \cup \D_2^2$ and at least one other irreducible curve.
So $\chib(\EC_2) \ge 3$, with equality if and only if $\EC_2$ is of type $\Irm_2$
in Kodaira's notation.

But, by the two previous paragraphs, we have
$$\chib(\pht^{-1}([0:1]))=9\qquad \text{and}\qquad \chib(\pht^{-1}([1:0])) = 8.$$
Since the sum of the Euler characteristics of the singular fibers is equal to $24$,
this forces $\chib(\EC_1)=\chib(\EC_1')=2$ and $\chib(\EC_2)=3$ and that there is no other
singular fiber for $\pht$.
This completes the proof of Theorem~\ref{theo:elliptic}(a).

\bigskip

\subsection{Local study at ${\boldsymbol{p_5}}$}
The fiber $\rhoh^{-1}(p_5)$ is the union of two smooth rational curves.
By~\eqref{eq:XCh5h1}, its intersection with $\XCh_5^{(h_1)}$ is given by the following equations:
$$\rhoh^{-1}(\phat_5) \cap \XCh_5^{(h_1)}=\{(a,b,c,d) \in \AM^4(\CM)~|~a=bd=b+8c=0\}.$$
It contains two irreducible components $\Delh_5^b$ and $\Delh_5^d$ which are respectively given
by the equations $a=b=c=0$ and $a=d=b+8c=0$. The restriction of $\phh$ to
$\Delh_5^b \setminus \{\phat_5\}$ is constant and equal to $[0:1]$,
while its restriction to $\Delh_5^d$ is constant and equal to $[1:0]$.
We denote by $\Delh_5^1$ (resp. $\Delh_5^2$) the closure of $\Delh_5^a$ (resp. $\Delh_5^b$)
in $\XCh$.

We denote by $\CCh_3$ and $\CCh_4$ the
respective strict tranforms of $\CC_3$ and $\CC_4$ in the blow-up $\XCh$.

\bigskip

\begin{lem}\label{lem:CCh3 CCh4}
With the above notation, the following holds:
\begin{itemize}
\itemth{a} The curve $\CCh_3$ meets $\Delh_5^1$ and does not meet $\Delh_5^2$.

\itemth{b} The curve $\CCh_4$ meets $\Delh_5^1 \cup \Delh_5^2$ only at $\phat_5$.
\end{itemize}
\end{lem}

\bigskip

\begin{proof}
Let $k \in \{3,4\}$.
Since $\CC_k$ is a smooth rational curve, its strict transform meets
$\rhoh^{-1}(p_5)$ (which is identified with the tangent cone of $\XC/W'$ at $p_5$)
at only one point (the line given by the tangent space of $\CC_k$ at $p_5$).

As $\CC_3\cap \UC_5$ is the subscheme of $(\XC/W') \cap \UC_5$ defined by the equations
$Y_3=h_1=h_2=h_4=0$, its strict transform does not meet the open chart $\XCh_5^{(h_1)}$.
So $\phat_5 \not\in \CCh_3$. So $\CCh_3$ meets $\Delh_5^1$ or $\Delh_5^2$ but does not
meet both. Since $\phh(\CCh_3 \setminus \{\phat_5\})=\phh(\Delh_5^1 \setminus \{\phat_5\})=
[0:1] \neq [1:0] = \phh(\Delh_5^2\setminus \{\phat_5\})$, we get~(a).

As $\CC_4 \cap \UC_5$ is the subscheme of $(\XC/W') \cap \UC_5$ defined by
the equations $Y_4=h_2=h_3=h_4=0$ and since $\rhoh(a,b,c,d)=(ab,acd^2,a,ac,ad,acd)$,
$\CCh_4 \cap \XCh_5^{(h_1)}$ is the subscheme of $\XCh_5^{(h_1)}$ defined by
the equations $c=d=0$. So it contains $\phat_5$.
\end{proof}

\bigskip

We denote by $\hat{\Delh}_5^1$ and $\hat{\Delh}_5^4$ the respective
strict transforms of $\Delh_5^1$
and $\Delh_5^2$ in $\hat{\XCh}$. Also, we denote by $\hat{\Delh}_5^2$ and $\hat{\Delh}_5^3$ the
two smooth rational curves contained in $\pih^{-1}(\phat_5)$, numbered in such a way that
$\hat{\Delh}_5^1 \cap \hat{\Delh}_5^2 \neq \vide$. Finally, we denote by
$\hat{\CCh}_4$ the strict transform of $\CCh_4$ in $\hat{\XCh}$. These are all
smooth rational curves. Using the equations of $\DCh$ given in~\eqref{eq:Dhat}, we get that:
\begin{itemize}
\item[$\bullet$] $\hat{\Delh}_5^1 \cap \DCh = \{((0,0,0,d),[0:0:1])~|~d \in \CM\}$ and
$$\pht((0,0,0,d),[0:0:1])=[0:1].$$

\item[$\bullet$] $\hat{\Delh}_5^2 \cap \DCh = \hat{\Delh}_5^2 = \{((0,0,0,0),[0:x:y]) ~|~[x:y] \in \PM^1(\CM)\}$
and
$$\pht((0,0,0,0),[0:x:y])=[(3125/108) x^2 : y^2].$$

\item[$\bullet$] $\hat{\Delh}_5^3 \cap \DCh = \hat{\Delh}_5^3 = \{((0,0,0,0),[x:y:0])~|~[x:y] \in \PM^1(\CM)\}$ and
$$\pht((0,0,0,0),[x:y:0])=[1:0].$$

\item[$\bullet$] $\hat{\Delh}_5^4 \cap \DCh = \{((0,-8c,c,0),[-8:1:0])~|~ c \in \CM\}$ and
$$\pht((0,-8c,c,0),[-8:1:0])=[1:0].$$

\item[$\bullet$] $\hat{\CCh}_4 = \{((a,b,0,0),[1:0:0])~|~
\frac{6561}{2048} ab^2 + 6 a - \frac{3125}{144} b=0\}$ and
$$\pht((a,b,0,0),[1:0:0])=[1:0].$$
\end{itemize}
This allows to check easily that the intersection graph given in Theorem~\ref{theo:elliptic}
is correct.

\bigskip

\subsection{Sections}
The fact that $\CCt_5^+$ and $\CCt_5^-$ are sections of $\pht$ follows from their
explicit description given in the proof of Proposition~\ref{prop:c5}. The fact that
$\D_5^2$ is a double section follows from the computations done at the end of the
previous subsection.

The proof of Theorem~\ref{theo:elliptic}(b) is now complete.

\bigskip

\section{Magma computations for the proof of Theorem~\ref{theo:elliptic}(c)~and~(e)}

\medskip

\subsection{Theorem~\ref{theo:elliptic}(c)}\label{sub:theo c}
In the next {\sc Magma} computation, we use the following numbering of the
nodes of the big connected subgraph of $(\bigstar)$:
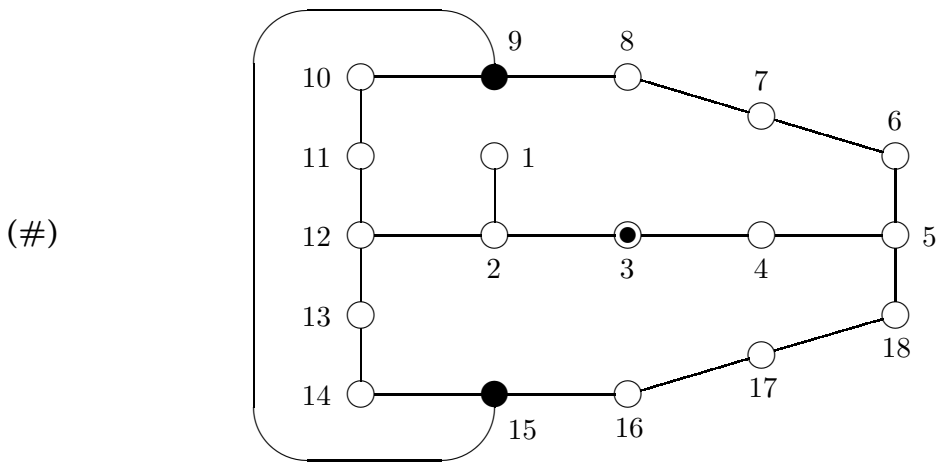
$$
\begin{picture}(290,170)
\put( 55,140){\circle{10}}\put(33,136){$10$}
\put(105,140){\circle*{10}}\put(110,150){$9$}
\put(155,140){\circle{10}}\put(152,150){$8$}
\put(205,125){\circle{10}}\put(202,135){$7$}
\put(255,110){\circle{10}}\put(252,120){$6$}
\put( 55,110){\circle{10}}\put(33,106){$11$}
\put(105,110){\circle{10}}\put(115,106){$1$}
\put( 55, 80){\circle{10}}\put(33,76){$12$}
\put(105, 80){\circle{10}}\put(102,64){$2$}
\put(155, 80){\circle{10}}\put(155,80){\circle*{6}}\put(152,64){$3$}
\put(205, 80){\circle{10}}\put(202,64){$4$}
\put(255, 80){\circle{10}}\put(265,76){$5$}
\put( 55, 50){\circle{10}}\put(33,46){$13$}
\put( 55, 20){\circle{10}}\put(33,16){$14$}
\put(105, 20){\circle*{10}}\put(110,3.3){$15$}
\put(155, 20){\circle{10}}\put(150,4){$16$}
\put(205, 35){\circle{10}}\put(200,19){$17$}
\put(255, 50){\circle{10}}\put(250,34){$18$}

\put( 60, 140){\line(1,0){40}}
\put( 55, 135){\line(0,-1){20}}
\put(110, 140){\line(1,0){40}}
\put( 55, 105){\line(0,-1){20}}
\put( 55, 75){\line(0,-1){20}}
\put( 55, 45){\line(0,-1){20}}
\put(105, 105){\line(0,-1){20}}
\put(255, 105){\line(0,-1){20}}
\put(255, 75){\line(0,-1){20}}
\put( 60, 80){\line(1,0){40}}
\put(110, 80){\line(1,0){40}}
\put(160, 80){\line(1,0){40}}
\put(210, 80){\line(1,0){40}}
\put( 60, 20){\line(1,0){40}}
\put(110, 20){\line(1,0){40}}

\qbezier(159.789,138.563)(180,132.5)(200.211,126.437)
\qbezier(209.789,123.563)(230,117.5)(250.211,111.437)
\qbezier(159.789,21.437)(180,27.5)(200.211,33.563)
\qbezier(209.789,36.563)(230,42.5)(250.211,48.437)

\put(60,140){\oval(90,50)[t]}
\put(60,20){\oval(90,50)[b]}
\put(15,140){\line(0,-1){120}}
\put(-78,78){${\boldsymbol{(\#)}}$}

\end{picture}
$$

\begin{comp}\label{comp:incidence}~
\begin{quotation}
{\small\begin{verbatim}
> G:=Graph<18 | [{2},{1,3,12},{2,4},{3,5},{4,6,18},{5,7},{6,8},
  {7,9},{8,10,15},{9,11},{10,12},{2,11,13},{12,14},{13,15},{9,14,16},
  {15,17},{16,18},{5,17}]>;
> M:=AdjacencyMatrix(G);
> M:=M-2*M^0;
> Rank(M);
16
> liste:=Subsets({1..18},16);
> liste:=[[j : j in k] : k in liste];
> minors:=[Minor(M,k,k) : k in liste];
> Gcd(minors);
19
> test:=[2,4,5,6,7,8,9,10,11,12,13,14,15,16,17,18];
> Minor(M,test,test);
-19
\end{verbatim}}
\end{quotation}
\end{comp}

\subsection{Theorem~\ref{theo:elliptic}(e)}\label{sub:theo e}
Let us first determine the image of $-q_P : P^\perp/P \longto \QM/2\ZM$.
Since
$$P=P_1 \mathop{\oplus}^\perp P_1' \mathop{\oplus}^\perp P_2 \mathop{\oplus}^\perp P_3,$$
with $P_1 = \ZM \D_1$, $P_1'=\ZM \D_1'$, $P_2=\ZM \D_2^1 \oplus \ZM \D_2^2$ and
$P_3$ is generated by the nodes numbered $2$, $4$, $5$, $6$,\dots, $18$ in the above
diagram~$(\#)$.
Let $M_0$ denote the submatrix of the matrix $M$ of Computation~\ref{comp:incidence}
corresponding to this list of $16$ nodes.
Set $\pi_1=\D_1/2$, $\pi_1'=\D_1'/2$, $\pi_2=(\D_2+2\D_2')/3$ and
$\pi_3$ be the vector of $\QM \otimes_\ZM P_3$ whose coordinates are given
by the fourth row of $M_0^{-1}$. Then
$$P^\perp/P = (\ZM/2\ZM) \pi_1 \oplus (\ZM/2\ZM) \pi_1' \oplus (\ZM/3\ZM) \pi_2
\oplus (\ZM/19\ZM) \pi_3.$$
Indeed, this follows from the fact that $\det(M_0)=-19$
and that some coordinates of $\pi_3$ have denominator equal to $19$:

\bigskip

\begin{comp}\label{comp:v3 coor}~
\begin{quotation}
{\small\begin{verbatim}
> M0:=Submatrix(M,test,test);
> M0:=ChangeRing(M0,Rationals());
> pi3:=Submatrix(M0^-1,[4],[1..16]);
> pi3[1][4];
-150/19
\end{verbatim}}
\end{quotation}
\end{comp}

\bigskip

The last command of Computation~\ref{comp:v3 coor} also shows that
$$\langle \pi_3,\pi_3 \rangle = -\frac{150}{19}.$$
Moreover,
$$\langle \pi_1,\pi_1 \rangle = \langle \pi_1',\pi_1'\rangle = -\frac{1}{2}
\qquad \text{and}\qquad \langle \pi_2,\pi_2 \rangle = -\frac{2}{3}.$$
Since $114 \pi \in P$ for all $\pi \in P^\perp$, we will compute the
image of the map $-114q_P$, which is a subset {\tt set0} of $\ZM/228\ZM$ (the last command
shows that {\tt set0} has cardinality $60$):

\bigskip

\begin{comp}\label{comp:set0}~
\begin{quotation}
{\small\begin{verbatim}
> set0:=&cat &cat &cat set0;
> set0:=Set([Numerator(114*s) mod 228 : s in set0]);
> # set0;
60
\end{verbatim}}
\end{quotation}
\end{comp}

\bigskip

Recall that
$$M_1=\begin{pmatrix} 2 & 0 \\ 0 & 114 \end{pmatrix},
\quad M_2=\begin{pmatrix} 6 & 0 \\ 0 & 38 \end{pmatrix},\quad
M_3=\begin{pmatrix} 4 & 2 \\ 2 & 58 \end{pmatrix}\quad\text{and}\quad
M_4=\begin{pmatrix} 12 & 6 \\ 6 & 22 \end{pmatrix}$$
and let $T_k$ denote the rank $2$ lattice determined by the matrix $M_k$.
We denote by $q_k : T_k^\perp/T_k \longto \QM/2\ZM$ the map induced
by the quadratic form on $T_k$.
Again, $114 \pi \in T_k$ for all $\pi \in T_k^\perp$, so we will in fact
compute the image {\tt setk} of $114 q_k$ (which is also a subset of $\ZM/228\ZM$):

\bigskip

\begin{comp}\label{comp:setk}~
\begin{quotation}
{\small\begin{verbatim}
> M1:=Matrix(Rationals(),2,2,[[2,0], [0,114]]);
> M2:=Matrix(Rationals(),2,2,[[6,0], [0,38]]);
> M3:=Matrix(Rationals(),2,2,[[4,2], [2,58]]);
> M4:=Matrix(Rationals(),2,2,[[12,6], [6,22]]);
>
> pi11:=Submatrix(M1^-1,[1],[1,2]);
> pi12:=Submatrix(M1^-1,[2],[1,2]);
> set1:=&cat [[(a*pi11+b*pi12)*M1*Transpose(a*pi11+b*pi12) :
      a in [0..113]] : b in [0..113]];
> set1:=Set([Numerator(114*s[1][1]) mod 228 : s in set1]);
>
> pi21:=Submatrix(M2^-1,[1],[1,2]);
> pi22:=Submatrix(M2^-1,[2],[1,2]);
> set2:=&cat [[(a*pi21+b*pi22)*M2*Transpose(a*pi21+b*pi22) :
      a in [0..113]] : b in [0..113]];
> set2:=Set([Numerator(114*s[1][1]) mod 228 : s in set2]);
>
> pi31:=Submatrix(M3^-1,[1],[1,2]);
> pi32:=Submatrix(M3^-1,[2],[1,2]);
> set3:=&cat [[(a*pi31+b*pi32)*M1*Transpose(a*pi31+b*pi32) :
      a in [0..113]] : b in [0..113]];
> set3:=Set([Numerator(114*s[1][1]) mod 228 : s in set3]);
>
> pi41:=Submatrix(M4^-1,[1],[1,2]);
> pi42:=Submatrix(M4^-1,[2],[1,2]);
> set4:=&cat [[(a*pi41+b*pi42)*M1*Transpose(a*pi41+b*pi42) :
      a in [0..113]] : b in [0..113]];
> set4:=Set([Numerator(114*s[1][1]) mod 228 : s in set4]);
> set0 eq set1;set0 eq set2;set0 eq set3;set0 eq set4;
true
false
false
false
\end{verbatim}}
\end{quotation}
\end{comp}

The last command shows that the image of $-q_P$ coincides with the image of $q_1$ and
does not coincide with the image of $q_2$, $q_3$ or $q_4$. This completes
the proof of Theorem~\ref{theo:elliptic}(e).

\end{document}